\documentclass[a4paper,11pt]{article}
\usepackage{amsfonts,amssymb,amsmath, dsfont,relsize,accents,amsthm,tikz}
\usepackage{bm}
\usepackage{mathrsfs}  % curly fonts
\usepackage{color}
\usepackage{upgreek}
\usepackage[english]{babel}
\usepackage{graphicx}
\usepackage{microtype}
\usepackage[colorlinks=true, pdfstartview=FitV, linkcolor=blue, citecolor=blue, urlcolor=blue,pagebackref=false]{hyperref}
\usepackage{graphicx}
\usepackage{epstopdf}

\definecolor{labelkey}{gray}{.8}
\definecolor{refkey}{gray}{.8}

\definecolor{darkred}{rgb}{0.9,0.1,0.1}

 \makeatletter
 \@addtoreset{equation}{section}
 \makeatother

 \makeatletter
 \@addtoreset{enunciato}{section}
 \makeatother

 \newcounter{enunciato}[section]

 \newtheorem{ittheorem}{Theorem}
 \newtheorem{itlemma}{Lemma}
 \newtheorem{itproposition}{Proposition}
 \newtheorem{itcorollary}{Corollary}
 \newtheorem{itdefinition}{Definition}
 \newtheorem{itremark}{Remark}
 \newtheorem{itclaim}{Claim}
 \newtheorem{itfact}{Fact}
 \newtheorem{itconjecture}{Conjecture}

 \newenvironment{theorem}{\addtocounter{enunciato}{1}
 \begin{ittheorem}}{\end{ittheorem}}

 \newenvironment{lemma}{\addtocounter{enunciato}{1}
 \begin{itlemma}}{\end{itlemma}}

 \newenvironment{proposition}{\addtocounter{enunciato}{1}
 \begin{itproposition}}{\end{itproposition}}

 \newenvironment{corollary}{\addtocounter{enunciato}{1}
 \begin{itcorollary}}{\end{itcorollary}}

 \newenvironment{definition}{\addtocounter{enunciato}{1}
 \begin{itdefinition}}{\end{itdefinition}}

 \newenvironment{remark}{\addtocounter{enunciato}{1}
 \begin{itremark}}{\end{itremark}}

 \newenvironment{claim}{\addtocounter{enunciato}{1}
 \begin{itclaim}}{\end{itclaim}}

 \newenvironment{fact}{\addtocounter{enunciato}{1}
 \begin{itfact}}{\end{itfact}}

 \newenvironment{conjecture}{\addtocounter{enunciato}{1}
 \begin{itconjecture}}{\end{itconjecture}}

 \newcommand{\be}[1]{\begin{equation}\label{#1}}
 \newcommand{\ee}{\end{equation}}

 \newcommand{\bl}[1]{\begin{lemma}\label{#1}}
 \newcommand{\el}{\end{lemma}}

 \newcommand{\br}[1]{\begin{remark}\label{#1}}
 \newcommand{\er}{\end{remark}}

 \newcommand{\bt}[1]{\begin{theorem}\label{#1}}
 \newcommand{\et}{\end{theorem}}

 \newcommand{\bd}[1]{\begin{definition}\label{#1}}
 \newcommand{\ed}{\end{definition}}

 \newcommand{\bcl}[1]{\begin{claim}\label{#1}}
 \newcommand{\ecl}{\end{claim}}

 \newcommand{\bfact}[1]{\begin{fact}\label{#1}}
 \newcommand{\efact}{\end{fact}}

 \newcommand{\bp}[1]{\begin{proposition}\label{#1}}
 \newcommand{\ep}{\end{proposition}}

 \newcommand{\bc}[1]{\begin{corollary}\label{#1}}
 \newcommand{\ec}{\end{corollary}}

 \newcommand{\bcj}[1]{\begin{conjecture}\label{#1}}
 \newcommand{\ecj}{\end{conjecture}}

 \newcommand{\bpr}{\begin{proof}}
 \newcommand{\epr}{\end{proof}}

 \newcommand{\bprlem}[1]{\begin{proofof}{\it Lemma \ref{#1}}.\,\,}
 \newcommand{\eprlem}{\end{proofof}}

 \newcommand{\bprthm}[1]{\begin{proofof}{\it Theorem \ref{#1}}.\,\,}
 \newcommand{\eprthm}{\end{proofof}}

 \newcommand{\bprprop}[1]{\begin{proofof}{\it Proposition \ref{#1}}.\,\,}
 \newcommand{\eprprop}{\end{proofof}}

 \newcommand{\bi}{\begin{itemize}}
 \newcommand{\ei}{\end{itemize}}

 \newcommand{\ben}{\begin{enumerate}}
 \newcommand{\een}{\end{enumerate}}

 \newcommand{\one}{{\mathchoice {1\mskip-4mu\mathrm l}
         {1\mskip-4mu\mathrm l}
         {1\mskip-4.5mu\mathrm l}
         {1\mskip-5mu\mathrm l}}}

\def \E {{\mathbb E}}

\def \N {{\mathbb N}}
\def \P {{\mathbb P}}

\def \R {{\mathbb R}}
\def \Z {{\mathbb Z}}

\def \lra \leftrightarrow

\def \ra {\rightarrow}
\def \ba {\begin{array}}
\def \ea {\end{array}}

\def \lra {\longrightarrow}
\def \bbe {\bar{\beta}}

\def \lra {{\leftrightarrow}}

\def \subset {\subseteq}

\def \scF {\mathscr{F}}
\def \rsa {\rightsquigarrow}

\def \tplus {t_{(+)}}
\def \tmin {t_{(-)}}

\def\one{\rlap{\mbox{\small\rm 1}}\kern.15em 1}

\newlength{\dhatheight}

\begin{document}
\title{The asymmetric multitype contact process}

\author{Thomas Mountford\textsuperscript{1}, Pedro Luis Barrios Pantoja$^2$, Daniel Valesin\textsuperscript{2}}
\footnotetext[1]{ Ecole Polytechnique F\'ed\'erale de Lausanne. \url{thomas.mountford@epfl.ch}}
\footnotetext[2]{University of Groningen. \url{p.l.barrios.pantoja@rug.nl}, \url{d.rodrigues.valesin@rug.nl}}
\date{October 17, 2017}
\maketitle

\begin{abstract}
In the multitype contact process, vertices of a graph can be empty or occupied by a type 1 or a type 2 individual; an individual of type $i$ dies with rate 1 and sends a descendant to a neighboring empty site with rate $\lambda_i$. We study this process on $\Z^d$ with $\lambda_1 > \lambda_2$ and $\lambda_1$ larger than the critical value of the (one-type) contact process. We prove that, if there is at least one type 1 individual in the initial configuration, then type 1 has a positive probability of never going extinct. Conditionally on this event, type 1 takes over a ball of radius growing linearly in time.  We also completely characterize the set of stationary distributions of the process and prove that the process started from any initial configuration converges to a convex combination of distributions in this set.
\end{abstract}

{\bf\large{}}\bigskip

%%%%%%%%%%%%%% Introduction %%%%%%%%%%%%%%%%%%%%%%%%%%%%%%%%%%%%%%%%%%%%%%%%%%%%%%%%%%%%%%%%%%%%%%%%%%%%%%%%%%%%%%%%%%%%%%

\section{Introduction}
The \textit{multitype contact process} is an interacting particle system introduced by Neuhauser in \cite{neuhauser} as a variant of Harris' contact process (\cite{harris}) and a model for biological competition between species occupying space. The model on the $d$-dimensional lattice $\Z^d$ is defined as  the continuous-time Markov process $(\xi_t)_{t \geq 0}$ on $\{0,1,2\}^{\Z^d}$ with infinitesimal pregenerator
\begin{equation}\label{eq:pregenerator}
\mathcal{L}f(\xi) = \sum_{x \in \Z^d} (f(\xi^{0 \to x}) - f(\xi)) + \sum_{i\in \{1,2\}} \lambda_i \sum_{\substack{x,y \in (\Z^d)^2:\\0 < \|x-y\| \leq R}}  \mathds{1}_{\{\xi(x) = 0,\;\xi(y) = i \}} \cdot  (f(\xi^{i\to x}) - f(\xi)),
\end{equation}
where 
$$\xi^{i\to x}(y) = \begin{cases} i,&\text{if } y = x;\\\xi(y),&\text{otherwise,}\end{cases}\qquad i \in \{0,1,2\},$$
The parameters $\lambda_1, \lambda_2 \geq 0$ are called the \textit{birth rates}, $R \in \N$ is called the \textit{range}, $\|\cdot\|$ is the $\ell^1$ norm on $\Z^d$, $\mathds{1}$ denotes the indicator function and $f: \{0,1,2\}^{\Z^d} \to \mathbb{R}$ is a local function.

Let us give the biological interpretation of the process and explain the dynamics in words. Each site $x \in \Z^d$ is a spatial location, which at any time $t$ can be empty ($\xi_t(x) = 0$) or occupied by an individual of type (or species) 1 or 2 ($\xi_t(x) =1$ or $2$). Individuals die with rate 1, leaving their site empty; additionally, an individual of type $i \in \{1,2\}$ at site $x$ attempts to create a descendant in each site $y$ with $0< \|x-y\| \leq R$ with rate $\lambda_i$; such a birth is only allowed if site $y$ is empty. It should be noted that, although here we take a single ``death rate'' equal to 1 and a single range equal to $R$, one could also define the model so that these parameters depend on the species.

Evidently, the multitype contact process has the feature that the ``all zero'' configuration is absorbing, as are both the sets of configurations
\begin{equation}\label{eq:a_sets}\mathscr{A}_1 = \{\xi: \xi(x) \neq 2 \;\forall x\},\qquad \mathscr{A}_2 = \{\xi: \xi(x) \neq 1 \;\forall x\}.\end{equation}
The process started from $\xi_0 \in \mathscr{A}_1$ is Harris' (one-type) contact process with interactions of range $R$ and rate $\lambda_1$. (Similarly, in the process started from $\xi_0 \in \mathscr{A}_2$, the 2's evolve as the 1's would evolve in a contact process with range $R$ and rate $\lambda_2$).

Whenever we want to emphasize that we are referring to the one-type, and not multitype, contact process, we will denote it by $(\zeta_t)_{t\geq 0}$. The contact process has been introduced in \cite{harris}; see \cite{lig85} and \cite{lig99} for a comprehensive exposition, and for all facts about the one-type contact process which we mention without giving an explicit reference. For the exposition in this introduction, the \textit{critical rate} of the one-type contact process will be relevant; this is defined as follows. Let $\P_{d,\lambda, R}$ be a probability measure under which the contact process $(\zeta_t)_{t\geq 0}$ on $\Z^d$ with rate $\lambda$ and range $R$ is defined. Note that the function $\lambda \mapsto \P_{d,\lambda,R}\left[\exists t: \zeta_t = \underline{0} \right]$ is non-increasing and let
$$\lambda_c = \lambda_c(d,R) = \sup\left\{\lambda: \; \P_{d,\lambda,R}\left[\exists t: \zeta_t \equiv \underline{0} \right] = 1 \right\}.$$ 
As is well known, $\lambda_c(d,R) \in (0,\infty)$ for every $d$ and $R$, and $\P_{d,\lambda_c,R}\left[\exists t: \zeta_t \equiv \underline{0} \right] = 1$. The set of (extremal) stationary distributions of the contact process consists of two measures: $\delta_{\underline{0}}$ (the unit mass on the ``all zero'' configuration) and $\bar{\mu}$, the limiting distribution, as time is taken to infinity, of the process started from the ``all one'' configuration. In case $\lambda \leq \lambda_c$, these two measures are equal; otherwise, $\bar{\mu}$ is a measure supported on configurations containing infinitely many 1's. The \textit{complete convergence theorem} for the contact process is the statement that, for any initial configuration $\zeta_0 \in \{0,1\}^{\Z^d}$,
$$\zeta_t \xrightarrow[\text{(d)}]{t\to\infty} \P[\exists t: \zeta_t = \underline{0}] \cdot  \delta_{\underline{0}} + \P[\nexists t: \zeta_t = \underline{0}] \cdot \bar{\mu}.$$

%In \cite{neuhauser}, the author studied the ergodic theory of the multitype contact process, that is, the set of stationary distributions and conditions under which the process converges to them, under various assumptions on the parameters.

In the multitype contact process $(\xi_t)$, we say that the 1's survive if the event
\begin{equation} \label{eq:survival_event}\mathcal{S}_1 = \left\{\nexists t:\; \xi_t \in \mathscr{A}_2\right\}\end{equation}
occurs; otherwise we say that the 1's go extinct. In studying extinction and survival, we must eliminate two trivial cases. First: in case there are infinitely many 1's in the initial configuration, it is easy to see that they survive almost surely. Second: if there are finitely many 1's in the initial configuration and $\lambda_1  \leq \lambda_c(d,R)$, then it is easy to see that the 1's almost surely go extinct (as then their evolution is stochastically dominated by that of a one-type contact process which almost surely reaches the ``all zero'' configuration). The references \cite{amp} and \cite{valesin} treat the multitype contact process for $d =1$ and the symmetric setting $\lambda_1 = \lambda_2$, and establish conditions for survival or extinction of one of the types (say, the 1's). Having eliminated the trivial cases above, we are left with the situation in which $\lambda_1 = \lambda_2 > \lambda_c(d=1,R)$ and $\xi_0$ only has finitely many 1's (so that the 1's are confined to an interval $[-m, m]$). It then turns out that the 1's almost surely go extinct if and only if they are surrounded by infinitely many 2's in both directions (that is, if $\xi_0(x) = 2$ for infinitely many $x < -m$ and infinitely many $x > m$). This result has been proved in \cite{amp} for $R = 1$, and in \cite{valesin} through different methods and for any $R$.

In this paper, we turn to the case of distinct rates and study survival of the type with larger rate (that is, we assume that $\lambda_1 > \lambda_2$ and study survival of the 1's). Our main result holds for any dimension and range.
\begin{theorem}
\label{thm:main} Let $d \geq 1$, $R \in \N$ and assume that $\lambda_1 > \lambda_2$ and $\lambda_1 > \lambda_c(d,R)$. \begin{enumerate} \item If $\xi_0$ is a configuration containing at least one type 1 individual, then the event $\mathcal{S}_1$ that the 1's survive has positive probability. \item There exists $\alpha > 0$ such that the following holds. If $\xi_0(0) = 1$ and $\xi_0(x) \neq 1$ for all $x \neq 0$, then conditioned on $\mathcal{S}_1$, almost surely there exists $t_0 \geq 0$ such that
\begin{equation} \xi_t(x) \in \{0,1\} \text{ for all } t \geq t_0 \text{ and } x \text{ with } \|x \|\leq \alpha t.\label{eq:main_cone}\end{equation}\end{enumerate}
\end{theorem}
Note the contrast (at least in dimension one) with the result of \cite{amp} and \cite{valesin} mentioned above. For instance, if $\lambda_1 > \lambda_2$, $\lambda_1 > \lambda_c$, $\xi_0(0) = 1$ and $\xi_0(x) = 2$ for all $x \neq 0$, then the 1's almost surely go extinct in the symmetric case $\lambda_1 = \lambda_2$ and survive with positive probability if $\lambda_1 > \lambda_2$.

Given a choice of the parameters $d,\lambda_1,\lambda_2, R$, let $\bar{\mu}_1$ and $\bar{\mu}_2$ be the limiting distributions for the process started from the ``all 1's'' and ``all 2's'' configurations, respectively. Evidently, for $i = 1,2$, $\bar{\mu}_i$ is supported on $\mathscr{A}_i$ and $\bar{\mu}_i \neq \delta_{\underline{0}}$ if and only if $\lambda_i > \lambda_c(d,R)$.
Also define the event
$$\mathcal{S}_2 = \{\nexists t: \xi_t \in \mathscr{A}_1\}.$$
We prove a complete convergence theorem for the asymmetric multitype contact process:
\begin{theorem}
\label{thm:main2} 
Assume that $\lambda_1 > \lambda_2$. For any $\xi_0 \in \{0,1,2\}^{\Z^d}$,
\begin{equation}\label{eq:thm_main2}
\xi_t \xrightarrow[(d)]{t \to \infty} \P(\mathcal{S}_1)\cdot \bar{\mu}_1 + \P(\mathcal{S}_1^c \cap \mathcal{S}_2) \cdot \bar{\mu}_2 + \P((\mathcal{S}_1 \cup \mathcal{S}_2)^c) \cdot \delta_{\underline{0}}.
\end{equation}
In particular, the set of extremal stationary distributions of the process is equal to $\{\delta_{\underline{0}},\bar{\mu}_1,\bar{\mu}_2\}$.
\end{theorem}
Note that the statement of the theorem includes the three situations: $\lambda_2<\lambda_1 \le \lambda_c$ (in which $\bar{\mu}_1 = \bar{\mu}_2 = \delta_{\underline{0}}$), $\lambda_2 \le \lambda_c < \lambda_1$ (in which $\bar{\mu}_2 = \delta_{\underline{0}}$, $\bar{\mu}_1 \neq \delta_{\underline{0}}$) and $\lambda_c < \lambda_2 < \lambda_1$ (in which $\bar{\mu}_1 \neq \delta_{\underline{0}}$ and $\bar{\mu}_2 \neq \delta_{\underline{0}}$).

In \cite{neuhauser}, the following weaker result is proved: if $\lambda_1 > \lambda_2 > \lambda_c$, and if $\xi_0$ is a random configuration whose distribution is translation invariant and contains 1's, then $\xi_t$ converges in distribution to $\bar{\mu}_1$. Note that under these assumptions, $\xi_0$ contains infinitely many 1's almost surely, so that $\P(\mathcal{S}_1) = 1$, so this is indeed a particular case of Theorem \ref{thm:main2}.

Let us explain the organization of the paper. Here is a scheme showing our main intermediate results and the dependence between them:\\

\begin{tikzpicture}
\draw [->] [thick] (2.8,0) --(4,0);
\draw [->] [thick] (4.9,-0.3) --(4.9,-0.7);
\draw [->] [thick] (4.9,-1.7) --(4.9,-1.3);
\draw [->] [thick] (6.8,-1) --(8,-1);
\draw [->] [thick] (10.8,-1) --(12,-0.5);
\draw [->] [thick] (10.8,-1) --(12,-1.5);
\node [right] at (0,0) {Proposition \ref{prop:cone}};
\node [right] at (4,0) {Lemma \ref{lem:cone_finally}};
\node [right] at (4,-2) {Proposition \ref{prop:steer}};
\node [right] at (4,-1) {Proposition \ref{prop:cone0_dual}};
\node [right] at (8,-1) {Proposition \ref{prop:cone0}};
\node [right] at (12,-0.5) {Theorem \ref{thm:main}};
\node [right] at (12,-1.5) {Theorem \ref{thm:main2}};
\end{tikzpicture}\\[.3cm]
The order in which we arrange these results  is somewhat convoluted:
\begin{itemize}
\item In Section \ref{s:prelim}, we introduce basic facts and definitions about the one-type and multitype contact process and their graphical constructions.
\item In Section \ref{s:proofst}, we state Proposition \ref{prop:cone0} and show how it is used (together with some other intermediate lemmas) to prove Theorems \ref{thm:main} and \ref{thm:main2}.
\item In Section \ref{s:ancestor}, we state Proposition \ref{prop:cone0_dual}, which is a modified version of Proposition \ref{prop:cone0}. We then state Proposition \ref{prop:steer} and Lemma \ref{lem:cone_finally} and, using these two results, we prove Proposition \ref{prop:cone0_dual}.
\item In Section \ref{s:times} and the first part of the Appendix, we prove Proposition \ref{prop:steer}.
\item In the second part of the Appendix, we state and prove Proposition \ref{prop:cone}, which implies Lemma \ref{lem:cone_finally}.
\end{itemize}

\section{Preliminaries on the one-type and multitype contact process}\label{s:prelim}
\subsection{One-type contact process}
Fix $d \in \N$, $R \in \N$ and $\lambda > 0$. A \textit{Harris system} for the contact process on $\Z^d$ with range $R$ and rate $\lambda$ is a family
\begin{equation}\label{eq:harris_system}
H = \left(\{D^x: x \in \Z^d\},\;\{D^{x,y}:x,y \in \Z^d,0<\|x-y\|\leq R\}\right),
\end{equation}
where each $D^x$ is a Poisson point process with rate 1 on $[0,\infty)$, each $D^{x,y}$ is a Poisson point process with rate $\lambda$ on $[0,\infty)$, and all these processes are independent (note that $D^{x,y} \neq D^{y,x}$). We view each $D^x$ and each $D^{x,y}$ as a discrete subset of $[0,\infty)$. When we have $t\in D^x$, we say that there is a \textit{death mark} at $(x,t)$; when we have $t \in D^{x,y}$, we say that there is an arrow from $(x,t)$ to $(y,t)$. We denote by $\P$ a probability measure in a probability space in which $H$ is defined.

The way in which a Harris system is used as a \textit{graphical construction} for the contact process is very well known, but let us present it in order to introduce the notation we will use. Points of the Poisson point processes $(D^x)$ and $(D^{x,y})$ are taken as instructions for the two types of transition in the dynamics:
\begin{align*}
&\text{if } t\in D^x,\; \text{ then } \zeta_t = \zeta_{t-}^{0 \to x};\\
&\text{if } t\in D^{x,y} \text{ and } \zeta_{t-}(x) = 1,\; \text{ then } \zeta_t = \zeta_{t-}^{1\to y}.
\end{align*}

In order to see how these rules and the initial configuration $\zeta_0$ determine the value of $\zeta_t(x)$ for any given $t$ and $x$, we use infection paths. Given $H$, an \textit{infection path} is a function $\gamma: I \to \Z^d$, where $I \subset [0,\infty)$ is an interval, satisfying the properties: for each $t \in I$, $t \notin D^{\gamma(t)}$ and $\gamma(t) \neq \gamma(t-)$ implies $t \in D^{\gamma(t-),\gamma(t)}$. This is often described in words as: an infection path may not touch death marks and may traverse arrows. In case $0 \leq s < t$ and there is an infection path $\gamma: [s,t] \to \Z^d$ with $\gamma(s) = x$ and $\gamma(t) = y$, we say that $(x,s)$ and $(y,t)$ \textit{are connected by an infection path}; we represent this with the notation $(x,s) \rsa (y,t)$. By convention, we say $(x,s) \rsa (x,s)$. We then have
$$\zeta_t(x) = \mathds{1}\{\exists y \in \Z^d: \zeta_0(y) = 1 \text{ and } (y,0) \rsa (x,t)\},\quad x\in \Z^d,\; t\geq 0.$$
The following is some additional notation we will use concerning infection paths. Given $A, B \subset \Z^d \times [0,\infty)$, we write $A \rsa B$ if there is an infection path connecting some $(x,s) \in A$ to some $(y,t) \in B$ (here we implicitly assume that $s \leq t$). In case $A = \{(x,s)\}$ (respectively, if $B = (y,t)$), we write $(x,s) \rsa B$ (respectively, $A \rsa (y,t)$) instead of $A \rsa B$. We write $(x,s) \rsa \infty$ if $(x,s) \rsa \Z^d \times \{t\}$ for every $t \geq s$.  We use the symbol $\not \rsa$ to express the negation of any of these statements (e.g. $(x,s) \not \rsa \infty$ if there is some $t$ for which $(x,s) \rsa \Z^d \times \{t\}$ does not hold). Given $\Lambda \subseteq \Z^d$, define
\begin{equation}
\label{eq:definition_death_time}
T^\Lambda = \sup\{t: \Lambda \times \{0\} \rsa \Z^d \times \{t\}\}.
\end{equation}

We will need some well-known estimates that hold in the supercritical regime, $\lambda> \lambda_c(d,R)$. First, there exist $b_1, b_2 > 0$ (depending on $d,R,\lambda$) such that
\begin{equation}\label{eq:1cp_moves_slowly}
\P\left[\exists (z,s) \in \Z^d \times [0,t]:\; \|z\| > x,\; (0,0) \rsa (z,s) \right] < \exp(b_1 t - b_2 x),\quad x > 0,\; t > 0.
\end{equation}
This follows from the proofs of Proposition 1.21 and Lemma 1.22 in Chapter I.1 of \cite{lig99}. Second, Theorem 2.30 in Chapter I.2 of \cite{lig99} states that there are constants $\bar{c}, {\bar{c}_1} > 0$ such that, for any $\Lambda \subset \Z^d$, $\Lambda \neq \varnothing$,
\begin{align}\label{eq:dies_quickly}
&\P\left[t < T^\Lambda < \infty \right] < e^{-\bar{c}t},\quad t > 0 \text{ and }\\[.2cm]
&\P\left[T^\Lambda < \infty \right] < e^{-{\bar{c}_1}\cdot \#\Lambda}.\label{eq:dies_many}
\end{align}
Third, there exists $\bar{c}_2 > 0$ such that
\begin{equation}\label{eq:no_cross}\begin{split}
&t > 0,\; \|x-y\| \leq \sqrt{t}\\& \Longrightarrow \quad \P\left[(x,0) \rsa \Z^d \times \{t\},\; \Z^d \times \{0\} \rsa (y,t),\; (x,0) \not\rsa (y,t) \right] < 1-\exp(-\bar{c}_2t).\end{split}
\end{equation}
This follows from standard arguments using the renormalization construction of Bezuidenhout and Grimmett, see \cite{bezui}. Since we could not find a reference for \eqref{eq:no_cross}, we give a rough sketch of proof. It suffices to prove the statement for $t$ large enough and for $x = 0$ and $y$ with $\|y\| \leq \sqrt{t}$. By the construction of \cite{bezui} and large deviations estimates of \cite{durschon}, there exist $\ell > 0$ and $\alpha > 0$ such that the following holds. Let $s_1 = t/2 -1$ and $s_2 = t/2$. Let $B_1,\ldots, B_N$ be an enumeration of the (disjoint) boxes of the form 
$$k(2\ell + 1)\cdot  e_1 + [-\ell,\ell]^d,\qquad k = \{-\lfloor \alpha t\rfloor, \ldots, \lfloor \alpha t \rfloor\}, $$
where $e_1$ is the first canonical vector of $\Z^d$ (note that the number of boxes, $N$, is of order $t$). Conditionally on $\{(0,0) \rsa \Z^d \times \{s_1\}\}$, with probability larger than $1-e^{-ct}$, we have
$$\#\{n \in \{1,\ldots, N\}: (0,0) \rsa B_n \times \{s_1\}\} \geq \frac{3N}{4}.$$
Conditionally on $\{\Z^d \times \{s_2\} \rsa (y,t)\}$, with probability larger than $1-e^{-ct}$, we have
$$\#\{n \in \{1,\ldots,N\}: B_n \times \{s_2\} \rsa (y,t)\} \geq \frac{3N}{4}.$$
If both these inequalities hold, there exists $I \subseteq \{1,\ldots, N\}$ with $\#I \geq \frac{N}{4}$ such that for each $n \in I$ there are $x_n, y_n \in B_n$ such that $(0,0) \rsa (x_n,s_1)$ and $(y_n,s_2) \rsa (y,t)$. If for some $n \in I$ we also have $(x_n,s_1) \rsa (y_n,s_2)$, we can then guarantee that $(0,0) \rsa (y,t)$. By insisting that the infection path connecting $(x_n,s_1)$ to $(y_n,s_2)$ stays inside $B_n \times [s_1,s_2]$, the availabilities of these infection paths are independent, and hence the number of $n \in I$ for which $(x_n,s_1) \rsa (y_n,s_2)$ dominates a Binomial($\#I,\delta$) random variable, for some $\delta > 0$. The desired statement \eqref{eq:no_cross} then follows from the fact that with high probability, such a binomial random variable is non-zero.

\subsection{Multitype contact process}

We now consider the multitype contact process on $\Z^d$ with range $R$ and rates $\lambda_1 > \lambda_2 > 0$, as given by the Markov pregenerator in \eqref{eq:pregenerator}. This process also admits a graphical construction, which we will represent as an \textit{augmented Harris system}, consisting of a pair $\mathbb{H} = (H,\mathcal{H})$ of two independent collections of Poisson point processes. The collection $H = (\{D^{x,y}\}, \{D^x\})$ is the same collection as the one given in \eqref{eq:harris_system}, with $\lambda$ replaced by $\lambda_2$ everywhere. We will continue referring to points of the sets $D^{x,y}$ as arrows and points of the sets $D^x$ as death marks. The second element of $\mathbb{H}$ is $$\mathcal{H} = \{\mathscr{D}^{x,y}:\; x,y \in \Z^d:\; 0 < \|x-y\| \leq R\},$$ a collection of independent Poisson point processes on $[0,\infty)$ with rate $\lambda_1 -\lambda_2$. We will refer to points of the sets $\mathscr{D}^{x,y}$ as \textit{selective arrows}. These will play the role of birth attempts that are only usable by type 1 individuals (whereas regular arrows are usable by both types). The rules through which these Poisson processes determine the evolution of $(\xi_t)_{t\geq 0}$ are:
\begin{align}
\label{eq:rule_multi1}&\text{if } t \in D^x,\text{ then } \xi_t = \xi_{t-}^{0\to x};\\[.2cm]
\label{eq:rule_multi2}&\text{if } t \in D^{x,y},\; \xi_{t-}(x) = i\text{ and }\xi_{t-}(y) = 0, \text{ then } \xi_t = \xi_{t-}^{i\to y},\; i = 1,2;\\[.2cm]
\label{eq:rule_multi3}&\text{if } t \in \mathscr{D}^{x,y},\; \xi_{t-}(x) = 1 \text{ and }\xi_{t-}(y) = 0,\text{ then }\xi_t = \xi_{t-}^{1 \to y}.
\end{align}
In the rest of the paper, we will assume that the dimension $d$, the range $R$ and the rates $\lambda_1, \lambda_2$ are fixed and define an augmented Harris system $\mathbb{H}$ from which the multitype contact process is defined. We will denote the probability measure in this probability space again by $\mathbb{P}$.

Since Theorem \ref{thm:main} assumes that $\lambda_1 > \lambda_c(d,R)$ and the statement of Theorem \ref{thm:main2} is trivial in case $\lambda_1 \leq \lambda_c(d,R)$, we adopt the following:\\[.2cm]
\noindent \textbf{Global assumption.} \textit{We always assume that $\lambda_1 > \lambda_2 > 0$ and that $\lambda_1 > \lambda_c(d,R)$. }\\[.2cm]
For many of the statements we make, it will be sufficient to give a proof under the more restrictive assumption that $\lambda_1 > \lambda_2 > \lambda_c$. Under this assumption, the `basic' Harris system $H$ already corresponds to a supercritical contact process. Although our assumptions on $\lambda_2$ will be stated explicitly, let us already mention here that from Section \ref{s:ancestor} onward, we assume that $\lambda_1 > \lambda_2 > \lambda_c$.

The notion of infection path introduced in the previous subsection will still be used here, but we now make a distinction between \textit{basic infection paths}  and \textit{selective infection paths}. \begin{definition} Basic infection paths (BIP's) are just the infection paths defined from $H$ as in the previous subsection; very importantly, their definition does not involve $\mathcal{H}$. Selective infection paths (SIP's) are defined as BIP's, with the difference that, in addition to the arrows (from $H$), they are also allowed to use the selective arrows (from $\mathcal{H}$). In other words, given an augmented Harris system $\mathbb{H} = (H,\mathcal{H})$, a selective infection path of $\mathbb{H}$ is a function $\gamma:[t_1,t_2]\to \Z^d$, where $0 \leq t_1 < t_2 \leq \infty$, so that
\begin{itemize}
\item $t \neq D^{\gamma(t)}$ for all $t$;
\item $\gamma(t) \neq \gamma(t-)$ implies $t \in D^{\gamma(t-),\gamma(t)} \cup \mathscr{D}^{\gamma(t-),\gamma(t)}$.
\end{itemize}\end{definition}
Of course, every BIP is also an SIP.

As before, the notation $(x,t_1) \rsa (y,t_2)$ indicates that there is a basic infection path from $(x,t_1)$ to $(y,t_2)$; \textit{we emphasize that this event involves $H$ but not $\cal{H}$}. The same goes for other types of events involving the symbol `$\rsa$', such as $A\times \{t_1\} \rsa (x,t_2)$, $(x,t_1) \rsa A \times \{t_2\}$, $(x,t) \rsa \infty$ etc. We will not employ any analogous notation to indicate that there is a selective infection path from one space-time point to another. The random variables $T^\Lambda$ from \eqref{eq:definition_death_time} are defined here in the same way, \text{making use of basic infection paths only}, and have no relation to $\cal{H}$.

Some simple consequences of the rules \eqref{eq:rule_multi1}-\eqref{eq:rule_multi3} are given by the following.
\begin{lemma}\label{lem:first_prop_paths} \textbf{(First properties of BIP's and SIP's)}  For any $t \geq 0$,
\begin{align}\label{eq:set_2_contained}
&\{x: \xi_t(x) = 2\} \subset \{x: \exists y \text{ with } \xi_0(y) = 2 \text{ and } (y,0) \rsa (x,t)\},\\[.2cm]
\label{eq:set_1_contained}&\{x: \xi_t(x) = 1\} \subset \{x: \exists y \text{ with } \xi_0(y) = 1 \text{ and there is an SIP from $(y,0)$ to $(x,t)$}\},\\[.2cm]
& \left\{\begin{array}{c} x: \exists y \text{ with } \xi_0(y) \neq 0\\ \text{ and }(y,0) \rsa (x,t)\end{array}\right\} \subseteq \{x: \xi_t(x) \neq 0\}  \subseteq \left\{\begin{array}{c}x: \exists y \text{ with } \xi_0(y) \neq 0\\ \text{ and there is an SIP}\\\text{ from $(y,0)$ to $(x,t)$} \end{array}\right\}.\label{eq:larger_for_zero} 
\end{align}
\end{lemma}
Note that the above inclusions do not allow one to fully determine the  state of the multitype contact process at a given time from $\xi_0$ and $\mathbb{H}$. Although it is possible to give such a characterization by introducing some more classes of paths, we will not need to do so.

\begin{proof}[Proof of Lemma \ref{lem:first_prop_paths}]
We start noting that, for any $(x,t) \in \Z^d \times [0,\infty)$, almost surely there exists $N = N(x,t)$ such that any (selective) infection path started anywhere in $\Z^d \times [0,t]$ and ending at $(x,t)$ has at most $N$ jumps. To see this, we observe that almost surely there exists $M = M(x,t)$ such that no (selective) infection path started outside $[x-M,x+M] \times [0,t]$ reaches $(x,t)$ (this can be shown using bound \eqref{eq:1cp_moves_slowly} and a time reversal argument; we omit the details). Next, note that the total number of points of all Poisson point processes (death marks, arrows, selective arrows) corresponding to sites or pairs of sites in $[x-M,x+M]$ and in the time interval $[0,t]$ is finite. This number is an upper bound for the number of jumps of any (selective) path to $(x,t)$.

Now, let us prove \eqref{eq:set_2_contained}. Fix $(x,t)$ such that $\xi_t(x) =2$. Define $$t_1 = \inf\{s \in [0,t]: \xi_s(x) = 2 \text{ on } [s,t]\}.$$
If $t_1= 0$, then we have $\xi_0(x) = 2$ and a BIP from $(x,0)$ to $(x,t)$ is given by $\gamma(s) = x$, $0 \leq s \leq t$. If $t_1 > 0$, then $\xi_{t_1-}(x) = 0$ and there exists $x_1 \in \Z^d$ with $0 < \|x-x_1\| \leq R$ such that $\xi_{t_1}(x_1) = 2$ and there is an arrow from $(x_1,t_1)$ to $(x,t_1)$. Then let 
$$t_2 =\inf\{s \in [0,t_1]: \xi_s(x_1) = 2 \text{ on } [s,t_1]\}.$$
In case $t_2 = 0$, then $\xi_0(x_1) = 2$ and a BIP from $(x_1,0)$ to $(x,t)$ is given by $\gamma = x_1 \cdot \mathds{1}_{[0,t_1)} + x \cdot \mathds{1}_{[t_1,t]}$. Otherwise we continue in this manner, defining $x_2$ and $t_3$ and so on; eventually the procedure must end with some $k$ such that $t_k = 0$ and $\xi_0(x_k) = 2$, otherwise we would obtain BIP's to $(x,t)$ with arbitrarily many jumps. The proof of \eqref{eq:set_1_contained} is the same. The second inclusion in \eqref{eq:larger_for_zero} follows from \eqref{eq:set_2_contained}, \eqref{eq:set_1_contained} and the fact that every BIP is an SIP.

The first inclusion in \eqref{eq:larger_for_zero} is easy to prove. Fix $(x,t)$ such that there is some $y$ with $\xi_0(y) \neq 0$ and $(y,0) \rsa (x,t)$. Fix a BIP from $(y,0)$ to $(x,t)$ and let $0< t_1< t_2<\cdots< t_k$ be the successive jump times of this path. It is then seen by induction that $\xi_{t_i}(\gamma(t_i)) \neq 0$ for each $i$ (note however that we could have $\xi_{t_i}(\gamma(t_i)) \neq \xi_0(y)$). It then follows that $\xi_t(x) \neq 0$.
\end{proof}

\begin{definition} A \emph{free basic infection path} (FBIP) is a basic infection path $\gamma:[t_1,t_2] \to \Z^d$ satisfying
\begin{equation}\label{eq:def_of_free}
s \in [t_1,t_2],\;\gamma(s) \neq \gamma(s-) \quad \Longrightarrow \quad \Z^d \times \{t_1\} \not \rsa (\gamma(s), s-).
\end{equation}
A \emph{free selective infection path} (FSIP) is a selective infection path satisfying \eqref{eq:def_of_free}. \end{definition}
Note that any FBIP is an FSIP. FBIP's satisfy the following important property. 
\begin{lemma}
\label{lem:unique_FBIP} \textbf{(Uniqueness property of FBIP's)}
For any $x \in \Z^d$ and $0 \leq s < t$, we either have $\Z^d \times \{s\} \not \rsa (x,t)$ or there is a unique FBIP from $\Z^d \times \{s\}$ to $(x,t)$.
\end{lemma}
This is proved in \cite{interface} (Lemma 2.4 in that paper), but let us present the idea of how to find the unique FBIP mentioned in the lemma. Finding it will be useful to understand some of the illustrative figures that appear in the rest of the paper. Assume $\Z^d \times \{s\} \rsa (x,t)$ and fix an arbitrary BIP $\upgamma: [s,t] \to \Z^d$ with $\upgamma(t) = x$. In case $\upgamma$ is not an FBIP, let $r$ be the largest time at which there is a jump violating the FBIP property, that is, so that $\upgamma(r-) \neq \upgamma(r)$ and $\Z^d \times \{s\} \rsa (\upgamma(r), r-)$. Then, there exists a BIP $\hat{\upgamma}: [s,r] \to \Z^d$ such that $\hat{\upgamma}(r-) = \hat{\upgamma}(r) =  \upgamma(r)$. Now, define a new BIP $\upgamma_1: [s,t] \to \Z^d$ by setting $\upgamma_1 = \hat{\upgamma} \cdot \mathds{1}_{[s,r)} + \upgamma \cdot \mathds{1}_{[r,t]}$. If $\upgamma_1$ is an FBIP, we are done. Otherwise, let $r_1$ be the largest time at which $\upgamma_1$ violates the FBIP property; we then have $r_1 <r $. We then repeat the above procedure, modifying $\upgamma_1$ in the same way we modified $\upgamma$, hence obtaining $\upgamma_2$, and then proceeding similarly to obtain $\upgamma_3$,$\upgamma_4$ etc. This procedure must eventually end at an FBIP because the BIP's in the sequence $\upgamma_1,\upgamma_2,\ldots$ are all distinct and there are only finitely many BIP's from $\Z^d \times \{s\}$ to $(x,t)$. 

We complement the list of facts in Lemma \ref{lem:first_prop_paths} with the following. Since the proof is very similar to that of \eqref{eq:larger_for_zero}, we omit it.
\begin{lemma} \textbf{(FSIP's carry 1's)} For any $t$,
\begin{equation}\label{eq:cond_ifsip}
\{x: \xi_t(x) = 1\} \supseteq \{x:\exists y \text{ with } \xi_0(y) = 1 \text{ and there is an FSIP from $(y,0)$ to $(x,t)$}\}.
\end{equation}
\end{lemma}

\begin{lemma}\label{lem:concatenation} \textbf{(Concatenation)}
If $\gamma_1: [t_1,t_2] \to \Z^d$ and $\gamma_2: [t_2,t_3] \to \Z^d$ are FSIP's  with $\gamma_1(t_2) = \gamma_2(t_2)$, then the path $\gamma: [t_1,t_3] \to \Z^d$ defined by $\gamma(t) = \gamma_1(t)$ for $t \in [t_1,t_2]$ and $\gamma(t) = \gamma_2(t)$ for $t \in [t_2,t_3]$ is also an FSIP. Moreover, if $\gamma_1$ and $\gamma_2$ are FBIP's, then $\gamma$ is a FBIP.
\end{lemma}
\begin{proof}
Assume $\gamma(t) \neq \gamma(t-)$ for some $t$. If $t \in [t_1,t_2]$, then $\Z^d \times \{t_1\} \not \rsa (\gamma(t), t-)$ since $\gamma_1$ is an FSIP. If $t \in [t_2,t_3]$, then $\Z^d \times \{t_2\} \not \rsa (\gamma(t),t-)$ since $\gamma_2$ is an FSIP, so $\Z^d \times \{t_1\} \not \rsa (\gamma(t),t-)$. The second statement is evident.
\end{proof}

\section{Proof of Theorems \ref{thm:main} and \ref{thm:main2}}\label{s:proofst}

We will now state a key result about infection paths that will allow us to prove our main results. Before doing so, let us introduce some notation for subsets of $\Z^d$ and of $\Z^d \times [0,\infty)$.
\begin{definition}\label{def:spacetime} Define the sets
\begin{align*}
&B_x(r) = \{y \in \Z^d:\|y-x\| \leq r\},\quad x \in \Z^d,\; r \geq 0;\\
&\mathscr{R}(x,t,\ell) = B_x(\ell)\times [t-\ell,t],\quad x \in \Z^d,\; t \geq 0,\;\ell \in [0,t];\\
&\mathscr{C}(x,t,\alpha) = \{(y,s) \in \Z^d \times [t,\infty): \|y-x\| \leq  \alpha (s-t)\},\quad x \in \Z^d,\; t,\alpha  \geq 0.
\end{align*}
\end{definition}
\begin{proposition}\label{prop:cone0}
Assume $\lambda_1 > \lambda_2 > \lambda_c$. There exists $\bar{c} > 0$ and $\bar{\beta} > 0$ such that the following holds. For any $s> r> \ell >0$ and $x,y \in \Z^d$ with $(x,s) \in \mathscr{C}(y,r,\bbe)$, we have
$$\P\left[\left\{\Z^d \times \{0\} \not \rsa (x,s) \right\} \cup \left\{\begin{array}{r} \exists (y',r') \in \mathscr{R}(y,r,\ell): \Z^d \times \{0\} \rsa (y',r')\\[.2cm] \text{and $\exists$ an FSIP from $(y',r')$ to $(x,s)$} \end{array} \right\}\right] > 1-\exp(-\bar{c}\ell). $$
\end{proposition}
See Figure \ref{fig:pair0} for a representation of the second event inside the probability.

\begin{figure}[htb]
\begin{center}
\setlength\fboxsep{0pt}
\setlength\fboxrule{0.5pt}
\fbox{\includegraphics[width = 0.5\textwidth]{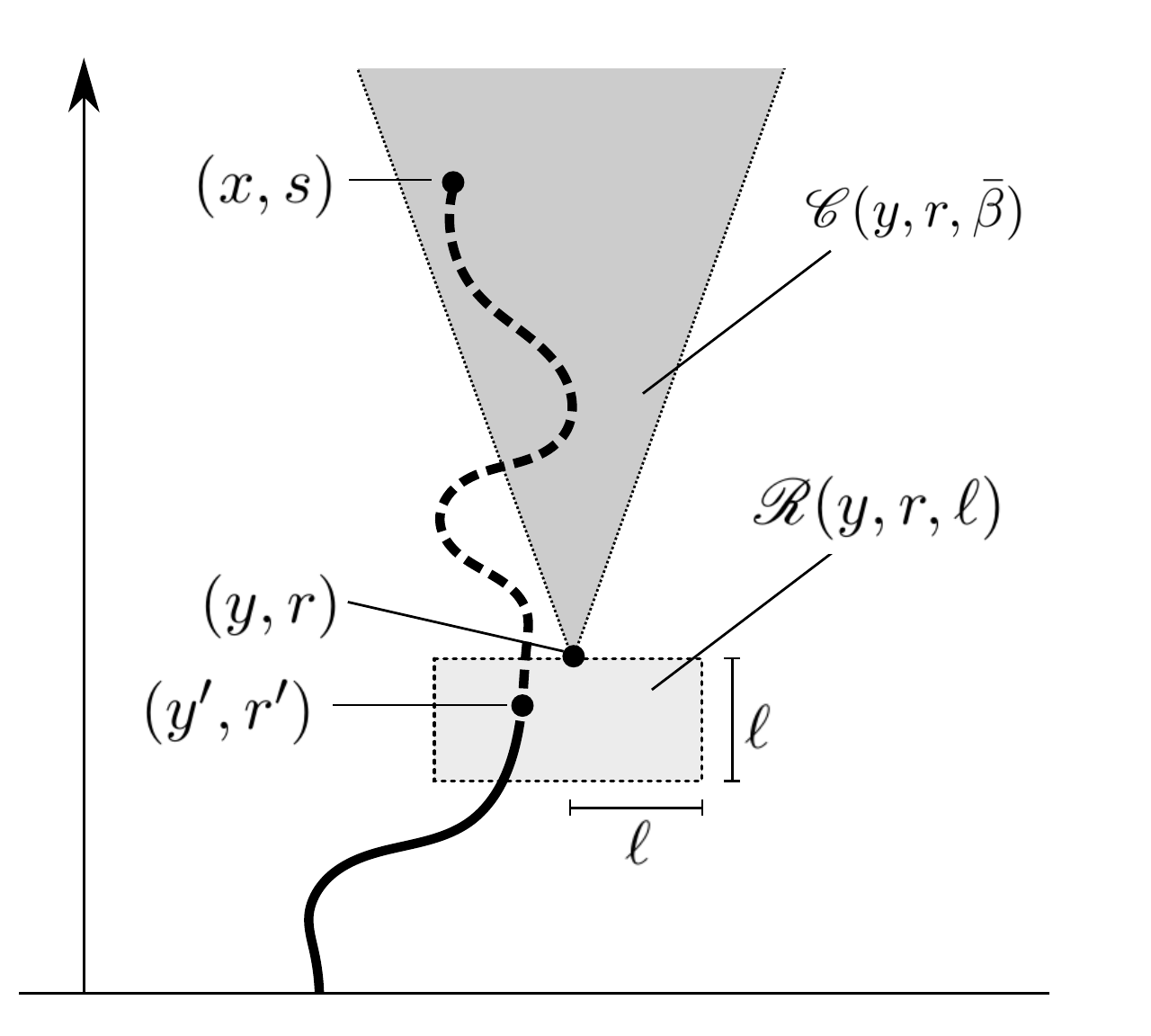}}
\end{center}
\caption{The event inside the conditional probability in Proposition \ref{prop:cone0} (in the $d = 1$ case). The thick black path is a basic infection path from some point in $\Z \times \{0\}$ to $(y',r')$. The dashed thick black path is a free selective infection path from $(y', r')$ to $(x,s)$. }
\label{fig:pair0}
\end{figure}

The proof of Proposition \ref{prop:cone0} will be carried out in stages in Sections \ref{s:ancestor}, \ref{s:times} and the Appendix. In the remainder of this section, we show how this proposition is used to prove our main theorems. 

We let $\bbe$ be as in Proposition \ref{prop:cone0} and define
\begin{equation*}
\bbe_k = \bbe\cdot 2^{-k},\; k \geq 1.
\end{equation*}

\begin{lemma}\label{lem:cone_integers}
For all $\varepsilon > 0$ there exists $m > 0$ such that, if $\xi_0 \equiv 1$ on $B_0(m)$, then 
\begin{equation}\label{eq:want_lemma_cone2}\P[\xi_t(x) \neq 2 \text{ for all $(x,t) \in \mathscr{C}(0,0,\bbe_1)$}] > 1-\varepsilon. \end{equation}
\end{lemma}
\begin{proof}
It suffices to prove the lemma under the assumption that $\lambda_1 > \lambda_2 > \lambda_c$, since reducing the value of $\lambda_2$ can only increase the probability on the left-hand side of \eqref{eq:want_lemma_cone2}. Additionally, by simple stochastic comparison considerations, it suffices to prove the lemma under the assumption that $\xi_0 \equiv 2$ outside $B_0(m)$. Together with $\xi_0 \equiv 1$ on $B_0(m)$, this gives
\begin{equation}\label{eq:never_zero} \xi_0(x) \neq 0 \text{ for all } x \in \Z^d,\end{equation}
which will be convenient.

The proof will rely on space-time sets whose definition will be based on an integer $\ell_0 > 0$.  We will assume that $\ell_0$ is taken as large as needed. Also, $c$ will be a small constant whose value may change from line to line.

We define $m = \ell_0^3$ and  
\begin{equation*}
r_0 = 0,\quad \ell_k = \ell_0 + k,\quad r_k =  \sum_{i=1}^k \ell_i^2,\quad k \in \N.
\end{equation*}
Next, define 
$$A_0 = \{(x,t) \in \Z^d \times  [0,r_1],\;\|x\| \leq \bbe t\}$$
and, for $k \geq 1$, define
\begin{align*}
&A_k = \{(x,t) \in \Z^d \times [r_k,r_{k+1}]: \|x\| \leq \bbe_1\cdot  r_k + \bbe\cdot (t-r_k)\},\\
&\bar{A}_k = \{(x,t) \in \Z^d \times [r_k,r_{k+1}]: \|x\| \leq \bbe_1\cdot  r_k + \bbe\cdot (t-r_k) + \ell_k\},\\
&I_k = \{(x,r_k): x \in \Z^d,\; \|x\| \leq \bbe_1 \cdot r_k + 2 \ell_k\} = B_0(\bbe_1 \cdot r_k + 2\ell_k) \times \{r_k\};
\end{align*}
see Figure \ref{fig:induct}. Note that
\begin{equation}\label{eq:include_Ak}
\mathscr{C}(0,0,\bbe_1) \subset \bigcup_{k=0}^\infty A_k.
\end{equation}

\begin{figure}[htb]
\begin{center}
\setlength\fboxsep{0pt}
\setlength\fboxrule{0.5pt}
\fbox{\includegraphics[width = 0.7\textwidth]{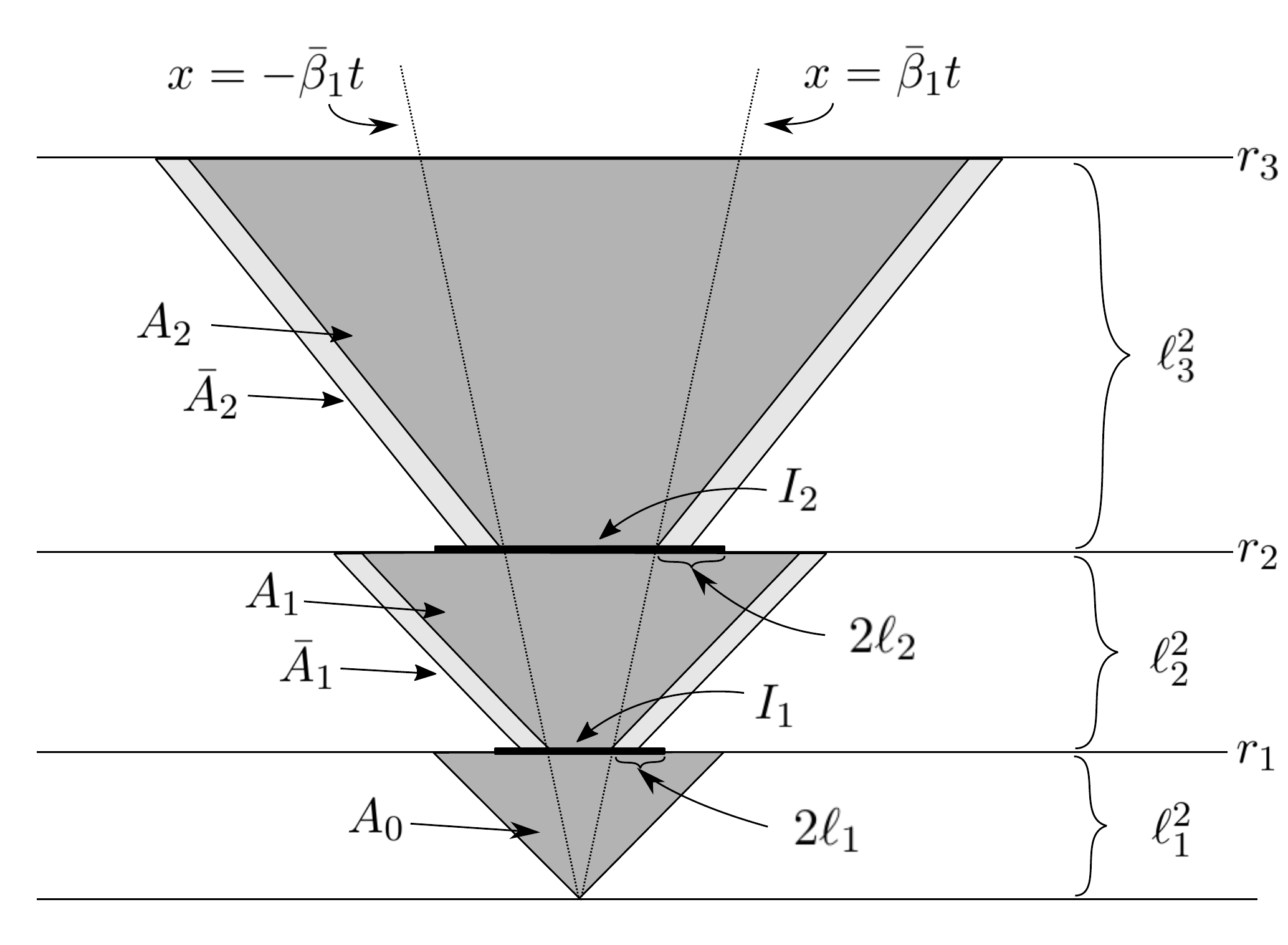}}
\end{center}
\caption{The sets $A_k$, $\bar{A}_k$ and $I_k$ in the proof of Lemma \ref{lem:cone_integers} in dimension one.}
\label{fig:induct}
\end{figure}

We claim that for any $k \geq 1$ and any $(x,s) \in \bar{A}_k$,
\begin{equation}\label{eq:claim_inside_proof}
\P\left[ \left\{\Z^d \times \{0\} \not\rsa (x,s)\right\} \cup \left\{\begin{array}{r}\exists (y',r') \in A_{k-1}:\;\Z^d \times \{0\} \rsa (y',r') \\ \text{ and $\exists$ an FSIP from $(y',r')$ to $(x,s)$} \end{array} \right\} \right]  > 1-\exp(-\bar{c}\ell_k),
\end{equation}
where $\bar{c}$ is the constant of Proposition \ref{prop:cone0}. Indeed, fix $(x,s) \in \bar{A}_k$. Using the definitions of $\bar{A}_k$ and $I_k$, it is easy to see that there exists $(y,r_k) \in I_k$ such that $(x,s) \in \mathscr{C}(y,r_k,\bbe)$. Moreover, using the fact that $(\bbe- \bbe_1)(r_k - r_{k-1}) \gg \ell_k$, we have $\mathscr{R}(y,r_k,\ell_k) \subset A_{k-1}$. Then, \eqref{eq:claim_inside_proof} follows directly from Proposition \ref{prop:cone0}.

We now define the events
\begin{align*}&E_0 = \{\xi_s(x) \neq 2 \text{ for all } (x,s) \in A_0\},\\[.2cm]& E_k = \{\xi_s(x) \neq 2 \text{ for all } (x,s) \in A_k\},\\[.2cm]
&\bar{E}_k = \{\xi_s(x) \neq 2 \text{ for all } (x,s) \in \bar{A}_k \text{ with } s \in \N\}, \quad k \geq 1.\end{align*}
We will show that if $\ell_0$ is large enough, there exists $c > 0$ such that
\begin{align}\label{eq:chain_E_0}
&\P(E_0) > 1-\exp(-c\ell_0),\\[.2cm]
&\label{eq:chain_E_1}\P(E_{k-1} \cap (\bar{E}_k)^c)< \exp(-c\ell_{k}),\; k\geq 1,\text{ and }\\[.2cm]
&\label{eq:chain_E_2}\P(\bar{E}_k \cap E_k^c) < \exp(-c\ell_k),\;k\geq 1.
\end{align}
These inequalities imply, for $\ell_0$ large enough, that
\begin{equation}\label{eq:big_intersection}\P\left(E_0 \cap \bigcap_{k=1}^\infty E_k\right) > 1-\varepsilon,\end{equation}
which by \eqref{eq:include_Ak} gives the desired result.

We start with \eqref{eq:chain_E_0}. By \eqref{eq:set_2_contained},
\begin{equation*} \begin{split}
\P(E_0^c) &\leq \sum_{y \in B_0(m)^c} \P\left[ (y,0) \rsa A_0\right]. 
\end{split}\end{equation*}
Since $A_0 \subset B_0(\bbe r_1) \times [0,r_1]$, for any $y$ in the above sum we have
\begin{align*}\P\left[(y,0) \rsa A_0\right] &\leq \P\left[(y,0) \rsa \left\{(x,s) \in \Z^d \times [0,r_1]: \|y-x\| \geq \|y\| - \bbe r_1\right\}\right]\\&\stackrel{\eqref{eq:1cp_moves_slowly}}{\leq}  \exp(b_1 r_1 - b_2 (\|y\|- \bbe r_1)) =  \exp(b_1 r_1 - b_2 (\|y\|- \bbe \ell_1^2)).\end{align*}
Since $m = \ell_0^3$, if $\ell_0$ is large enough and $c$ is small enough, \eqref{eq:chain_E_0} follows.

We now deal with \eqref{eq:chain_E_1}. For each $k \geq 1$ and each $(x,s) \in \bar{A}_k$, let $\bar{E}_k(x,s)$ be the event inside the probability in \eqref{eq:claim_inside_proof}, that is,
\begin{equation*} \bar{E}_k(x,s)= \{\Z^d \times \{0\} \not \rsa (x,s)\} \cup \left\{\begin{array}{r}\exists (y',r') \in A_{k-1}:\;\Z^d \times \{0\} \rsa (y',r') \text{ and  }\\ \text{$\exists$ a FSIP from $(y', r')$ to $(x,s)$ }  \end{array}\right\}.\end{equation*}
We claim that, for all $k \geq 1$,
\begin{equation}\label{eq:aux_with_prime} E_{k-1} \cap \left(\bigcap_{\substack{(x,s) \in \bar{A}_k:\\s\in \N}} \bar{E}_k(x,s)\right) \subset E_{k-1} \cap \bar{E}_k.\end{equation}
Indeed, assume that the event on the left-hand side occurs and fix $(x,s) \in \bar{A}_k$ with $s \in \N$; we have to prove that $\xi_s(x) \in \{0,1\}$. First assume that $\Z^d \times \{0\} \not\rsa (x,s)$, that is, there is no BIP from $\Z^d \times \{0\}$ to $(x,s)$. Then, by \eqref{eq:set_2_contained}, we have $\xi_s(x) \neq 2$ as desired. Now assume that  $\Z^d \times \{0\} \rsa (x,s)$; since $\bar{E}_k(x,s)$ occurs, there exists some $(y',r') \in A_{k-1}$ such that \begin{equation}\label{eq:very_aux0}\Z^d \times \{0\} \rsa (y',r')\end{equation} and \begin{equation} \label{eq:very_aux1}\text{there exists an FSIP from $(y',r')$ to $(x,s)$.}\end{equation} Now, \eqref{eq:never_zero}, \eqref{eq:very_aux0} and \eqref{eq:larger_for_zero} give $\xi_{r'}(y') \neq 0$. Then, since $(y',r') \in A_{k-1}$ and we are also under the assumption that $E_{k-1}$ occurs, we get $\xi_{r'}(y') = 1$. Then, \eqref{eq:cond_ifsip} and \eqref{eq:very_aux1} give $\xi_s(x) = 1$. This proves \eqref{eq:aux_with_prime}. We thus have
\begin{align}\label{eq:estimate_E_prime}
\P(E_{k-1} \cap (\bar{E}_k)^c) \leq \sum_{\substack{(x,s) \in \bar{A}_k:\\s \in \N}} \P((\bar{E}_k(x,s))^c).
\end{align}
It follows from \eqref{eq:claim_inside_proof} that, for any $(x,s) \in \bar{A}_k$,
\begin{equation*}
\P(\bar{E}_k(x,s)) > 1- \exp(-\bar{c}\ell_{k}).
\end{equation*}
Moreover, since $\bar{A}_k \subset B_0(\bbe r_k) \times [r_{k},r_{k+1}]$,
$$\#\{(x,s) \in \bar{A}_k: s \in \N\} \leq (2\bbe r_k)^{d} \cdot r_{k+1} \leq \ell_k^{10d} $$
if $\ell_0$ (and hence $\ell_k$) is large enough. Using these estimates in \eqref{eq:estimate_E_prime} gives \eqref{eq:chain_E_1}.

Finally, we turn to \eqref{eq:chain_E_2}:
\begin{align}\label{eq:very_aux_333}
\P(\bar{E}_k\cap E_k^c) \leq \sum_{j = \lfloor r_k \rfloor}^{\lfloor r_{k+1} \rfloor} \sum_{\substack{y \in \Z^d:\\(y,j) \notin \bar{A}_k}} \P\left[(y,j) \rsa \{(x,s) \in A_k: s \in [j,j+1]\} \right].
\end{align}
Fix $j \in \{\lfloor r_k\rfloor,\lfloor r_k\rfloor+1,\ldots,\lfloor r_{k+1}\rfloor\}$, $y \in \Z^d$ with $(y,j) \notin \bar{A}_k$ and $(x,s) \in A_k$ with $s \in [j,j+1]$. Letting $a_{k,j} = \bbe_1 r_k + \bbe(j-r_k)$, note that
$$(y,j) \notin \bar{A}_k \Longrightarrow \|y \| \geq a_{k,j} + \ell_k,\qquad (x,s) \in A_k \Longrightarrow \|x\| \leq \bbe_1 r_k + \bbe(j+1-r_k) =  a_{k,j} + \bbe;$$
from the second implication it follows that
$$\|y-x\| \geq | \|y\| - \|x\|| \geq  \|y\| - a_{k,j} - \bbe.$$ 
This shows that 
\begin{align*}&\P\left[(y,j) \rsa \{(x,s) \in A_k: s \in [j,j+1]\} \right] \\
&\leq \P\left[(y,j) \rsa \{(x,s): s \in [j,j+1],\; \|x-y\| \geq \|y\| - a_{k,j} + \bbe \} \right]\\
&\stackrel{\eqref{eq:1cp_moves_slowly}}{\leq} \exp\left(b_1 - b_2(\|y\| - a_{k,j} - \bbe)\right).\end{align*}
Using this bound in \eqref{eq:very_aux_333}, we obtain
$$\P(\bar{E}_k\cap E_k^c) \leq \sum_{j = \lfloor r_k \rfloor}^{\lfloor r_{k+1} \rfloor} \sum_{\substack{y \in \Z^d:\\\|y\| \geq a_{k,j} + \ell_k}}\exp\left(b_1 - b_2(\|y\| - a_{k,j} - \bbe)\right) < \exp(-c\ell_k)$$
for some $c > 0$.
\end{proof}

\iffalse \begin{corollary}\label{cor:cone0}
There exists $\bar{\beta}_2 > 0$ such that the following holds. For all $\varepsilon > 0$ there exists $m > 0$ such that, if $\xi_0 \equiv 1$ on $[-m,m]$, then 
\begin{equation}\label{eq:want_lemma_cone2}\P[\xi_t(x) \neq 2 \text{ for all $t\geq 0$ and $x \in \Z$ such that $|x| \leq \bar{\beta}_2 \cdot t$}] > 1-\varepsilon. \end{equation}
\end{corollary}
\begin{proof}
By \eqref{eq:1cp_moves_slowly} and Lemma \ref{lem:cone_integers}, by choosing $m$ large, with probability arbitrarily close to 1 we can guarantee that, if $\xi_0 \equiv 1$ on $[-m,m]$, then 
\begin{equation}\xi_t(x) \neq 2 \text{ for all } (x,t) \text{ with } [t \leq \sqrt{m},\;|x| \leq m/2] \text{ or } [t \geq \sqrt{m},\; t \in \N,\; |x| \leq \bbe_2 \cdot t].\label{eq:must_int}\end{equation}
Assume that \eqref{eq:must_int} is the case and that, for some $s \geq 0$ and $x$ with $|x| \leq \bbe_3 s$, we have $\xi_s(x) = 2$. If $m$ is large enough,
$$\{(x,t): t \leq \sqrt{m},\; |x| \leq \bbe_3 t\} \subset [-m/2,m/2] \times [0,\sqrt{m}]$$
and we have no 2's in the latter set, so we must have $s \geq \sqrt{m}$. Let $s' \in \N$ be such that $s \in [s',s'+1]$. Since we have
$$\xi_{s'}(y) \neq 2\; \forall y \in [-\bbe_2 s', \bbe_2 s'], $$
the only way $\xi_s(x) = 2$ is possible is if there is some $y$ with $|y| \geq \bbe_2 s'$ so that $(y, s') \rsa (x,s)$.
\end{proof}
\fi

\begin{definition} Define the set of configurations
\begin{equation*}
G(n,u_1,u_2) = \left\{\begin{array}{ll}\xi \in \{0,1,2\}^{\Z^d}:& \xi(x) \neq 2 \text{ for all $x \in B_0(u_2)$,}\\[.2cm] &\#\{x \in B_0(u_1):\;\xi(x) = 1 \} > 
n \end{array} \right\}, \;n \in \N,\;u_1, u_2 > 0.
\end{equation*}
\end{definition}
Note that
\begin{equation}\label{eq:comparisonn}
n \geq n',\; u_1 \leq u_1',\; u_2 \geq u_2' \quad \Longrightarrow \quad G(n,u_1,u_2) \subseteq G(n',u_1',u_2').
\end{equation}

Recall the definition of $\mathcal{S}_1$ in \eqref{eq:survival_event}. 
\begin{lemma}
\label{lem:Lchange1}
For all $\varepsilon > 0$ and $n \in \N$ there exists $m > 0$ such that
\begin{equation*} \label{eq:aux_aux_m}\xi_0 \equiv 1 \text{ on } B_0(m) \quad \Longrightarrow \quad \P\left[\xi_t \in G(n,m,\bbe_1 t)\right] > 1-\varepsilon \text{ for all } t \geq m^3. \end{equation*}
In particular, for any $\varepsilon > 0$ there exists $m > 0$ such that
\begin{equation}\label{eq:aux_aux_mm}
\xi_0 \equiv 1 \text{ on }B_0(m) \quad \Longrightarrow \quad \P(\mathcal{S}_1) > 1-\varepsilon.
\end{equation}
\end{lemma}
\begin{proof}
It suffices to prove the statements under the assumption that $\lambda_1 > \lambda_2 > \lambda_c$.

We claim that, if $n$ is fixed, $m$ is then taken large enough, $\xi_0$ is identically one on $B_0(m)$ and $t \geq m^3$, then the following four events occur with high probability:
\begin{align*}
&A_1 = \{\xi_t(x)\neq 2 \text{ for all } x \in B_0(\bbe_1 t)\};\\
&A_2 = \{\exists y \in B_0(m): (y,0) \rsa \Z^d \times \{t\}\};\\
&A_3 = \{\# \{x \in B_0(m): \Z^d \times \{0\} \rsa (x,t)\} > n \};\\
&A_4 = \{\text{for any } x, y \in B_0(m),\text{ if } (y,0) \rsa \Z^d \times \{t\} \text{ and } \Z^d \times \{0\} \rsa (x,t),\text{ then } (y,0) \rsa (x,t)\}.
\end{align*}
To see that $A_1, A_2$ and $A_4$ hold with high probability when $m$ is large enough and $t \geq m^3$, respectively apply Lemma \ref{lem:cone_integers}, \eqref{eq:dies_many} and \eqref{eq:no_cross}. For $A_3$, note that under $\mathbb{P}$ the set $\{x \in B_0(m): \Z^d \times \{0\} \rsa (x,t)\}$ is stochastically decreasing in $t$ (since it has the same distribution as $\{x \in B_0(m): (x,0) \rsa \Z^d \times \{t\}\}$), 
 and hence
\begin{align*}\P(A_3) &\geq \lim_{s\to\infty} \P\left[\#\{x \in B_0(m): \Z^d \times \{0\} \rsa (x,s)\} > n \right] \\&= \mu_1'(\{\xi: \#\{x \in B_0(m): \xi(x) = 1\} > n\}),\end{align*}
where $\mu_1'$ is the upper stationary distribution of a one-type contact process with rate $\lambda_2$ (rather than $\lambda_1$). Now, since $\mu_1'$ is supported on configurations with infinitely many 1's, we can choose $m$ so that the right-hand side is arbitrarily close to 1.

Suppose now that the four events occur. Fix $x \in B_0(m)$ such that $\Z^d \times \{0\} \rsa (x,t)$. Since $A_2$ occurs, we can take $y \in B_0(m)$ such that $(y,0) \rsa \Z^d \times \{t\}$; then, since $A_4$ occurs, we have $(y,0) \rsa (x,t)$. Using the first inclusion in \eqref{eq:larger_for_zero} and the fact that $\xi_0(y) = 1$, we obtain $\xi_t(x) \neq 0$. Since $A_1$ occurs, we then have $\xi_t(x) = 1$.

The second statement of the lemma follows from observing that $\mathcal{S}_1 = \cap_{t \geq 0} \{\exists x: \xi_t(x) = 1\}$.
\end{proof}

\begin{proof}[Proof of Theorem \ref{thm:main}]
We will need the fact:
\begin{equation}
\label{eq:for_all_M_delta}\begin{split}
&\forall m > 0\; \exists \delta > 0:\;\P\left[\xi_{t+1} \equiv 1 \text{ on } B_x(m) \mid \xi_t(x) = 1 \right] > \delta,\quad (x,t) \in \Z^d\times [0,\infty).\end{split}
\end{equation}
This follows from the fact that, if $\xi_t(x) = 1$, then $\xi_{t+1} \equiv 1$ on $B_x(m)$ can be achieved from finitely many prescription on the Poisson processes of the Harris system on the space-time set $B_x(m) \times [t,t+1]$. In fact, by using several disjoint space-time sets of this form, we can also show that
\begin{equation}\label{eq:for_all_M_eps}\begin{split}
&\forall \varepsilon > 0\;\forall m > 0\; \exists n > 0:\;\P\left[\exists x:\;\xi_{t+1}\equiv 1 \text{ on } B_x(m) \mid \#\{x:\xi_t(x) = 1\} \geq n \right] > 1-\varepsilon.
\end{split}
\end{equation}

By simple monotonicity and translation invariance considerations, to prove the first statement of Theorem \ref{thm:main}, it suffices to prove that $\P(\mathcal{S}_1) > 0$ for the case where $\xi_0$ is the configuration defined by $\xi_0(0) = 1$ and $\xi_0(x) = 2$ for all $x \neq 0$. But this is an immediate consequence of \eqref{eq:aux_aux_mm} and \eqref{eq:for_all_M_delta}.

We now turn to the second statement of the theorem. 
We start noting that, for any $n > 0$,
\begin{equation}\label{eq:drop_below_n}\P\left[\mathcal{S}_1 \cap \left\{\liminf_{t\to\infty} \#\{x: \xi_t(x) = 1\} < n\right\} \right] = 0.\end{equation}
This follows from elementary considerations concerning absorption probabilities of Markov processes: each time we have $\#\{x: \xi_t(x) = 1\} < n$, there is a positive chance $\delta_n > 0$ that, in the next second, all the 1's die without giving birth; hence, if the 1's are to survive, the population of 1's cannot drop below $n$ infinitely many times.

Now, \eqref{eq:for_all_M_eps} and \eqref{eq:drop_below_n} together imply that, for all $m > 0$,
\begin{equation*}
\P\left[\mathcal{S}_1 \backslash \left\{ \exists (x,t): \xi_t \equiv 1 \text{ on } B_x(m) \right\} \right] = 0
\end{equation*}
Together with Lemma \ref{lem:cone_integers}, this gives
\begin{equation}\label{eq:last_overlap}
\P\left[\exists (x,t): \xi_s(y) \neq 2 \text{ for all } (y,s) \text{ with } s \geq t,\; \|y-x\| \leq \bbe_1 (s-t)\mid \mathcal{S}_1 \right] = 1.
\end{equation}
Now, note that for any $(x,t) \in \Z^d \times [0,\infty)$, there exists $t' > 0$ such that
$$\{(y,s): s\geq t,\; \|y-x\| \leq \bbe_1(s-t)\} \supset \{(y,s): s \geq t',\; \|y\| \leq \bbe_2\cdot s\}.$$
Hence, \eqref{eq:last_overlap} gives the desired result, with $\alpha = \bbe_2$.
\end{proof}

\begin{lemma}\label{lem:many_survival_goes}
For all $\varepsilon > 0$ there exists $n \in \N$ such that, if $\#\{x: \xi_0(x) = 1\} \geq n$, then $\P(\mathcal{S}_1) > 1-\varepsilon$.
\end{lemma}
\begin{proof}
The statement is an immediate consequence of \eqref{eq:aux_aux_mm} and \eqref{eq:for_all_M_eps}.
\end{proof}

\iffalse\begin{lemma}
\label{lem:Lchange2}
Let $\xi_0 \in \{0,1,2\}^{\Z^d}$ be a configuration with at least one site in state 1. For all $\varepsilon > 0$ and $m > 0$ there exists $\uptau > 0$ and $\ell > 0$ such that
$$\P\left[ \exists x \in B_0(\ell):\; \xi_\tau \equiv 1 \text{ on } B_x(m)\mid \mathcal{S}_1\right] > 1-\varepsilon.$$
\end{lemma}
\begin{proof}
aaa
\end{proof}\fi

\begin{lemma}\label{lem:Lchange3}
Let $\xi_0 \in \{0,1,2\}^{\Z^d}$ be a configuration with at least one site in state 1. For all $\varepsilon > 0$ and $n > 0$ there exists $s_0$ and $r_0$ such that
\begin{equation*}
s \geq s_0 \quad \Longrightarrow \quad \P\left[ \xi_s \in G(n,r_0,\bbe_2 s) \mid \mathcal{S}_1\right] > 1-\varepsilon.
\end{equation*}
\end{lemma}
\begin{proof} Fix $\xi_0, \varepsilon, n$ as in the statement of the lemma. Choose $m$ corresponding to $\varepsilon$ and $n$ in the first part of Lemma \ref{lem:Lchange1}. Using \eqref{eq:for_all_M_eps} and \eqref{eq:drop_below_n}, it is easy to see that there exist $t_0 > 0$ and $\ell_0 > 0$ such that, defining
\begin{equation*}
E_1 = \left\{\exists x_0 \in B_0(\ell_0):\;\xi_{t_0} \equiv 1 \text{ on } B_{x_0}(m)\right\},
\end{equation*}
we have $\P(E_1 \mid \mathcal{S}_1) > 1-\varepsilon$. Next, defining
$$E_2(t) = \left\{\begin{array}{ll}\exists x_0 \in B_0(\ell_0):& \xi_t(y) \neq 2 \;\forall y \in B_{x_0}(\bbe_1(t-t_0)),\\[.2cm]& \# \{y \in B_{x_0}(m): \xi_t(y) = 1\} > n \end{array} \right\},\quad t \geq t_0 + m^3,$$
the choice of $m$ implies in $\P(E_2(t) \mid E_1) > 1-\varepsilon$ for all $t \geq t_0 + m^3$.
Hence,
$$\P(\mathcal{S}_1 \cap E_2(t)^c) \leq \P(\mathcal{S}_1 \cap E_1^c) + \P(E_1 \cap E_2(t)^c) \leq 2\varepsilon\quad \Longrightarrow \quad \P(E_2(t) \mid \mathcal{S}_1) \geq 1- \frac{2\varepsilon}{\P(\mathcal{S}_1)}.$$
To conclude, choose $r_0 > \ell_0 + m$ and choose $s_0$ large enough that $\bbe_1(s_0 - t_0) > \bbe_2 s_0 + \ell_0$, so that
$$x_0 \in B_0(\ell_0),\;s \geq s_0\quad \Longrightarrow \quad B_{x_0}(m) \subset B_0(r_0),\;B_{x_0}(\bbe_1(s-t)) \supset B_0(\bbe_2 s).$$
Due to these inclusions, for any $s \geq s_0$ we have $E_2(s) \subset \{\xi_s \in G(n,r_0,\bbe_2 s)\}$.
\end{proof}

\begin{lemma}\label{lem:Lchange4}
Let $f: \{0,1,2\}^{\Z^d} \to \mathbb{R}$ be a function depending only on finitely many coordinates. For all $\varepsilon > 0$ there exists $n_0 \in \N$ and $u_0 > 0$ such that
\begin{equation*}
n \geq n_0,\;u \geq u_0,\;\xi_0 \in G(n,\sqrt{u},u^2)\quad \Longrightarrow \quad \left| \E[f(\xi_u)] - \int fd\mu_1\right| < \varepsilon.
\end{equation*}
\end{lemma}
\begin{proof}
Fix $\varepsilon > 0$. Let $n \in \N$ and $u > 0$ (throughout the proof, we will assume that $n$ and $u$ are large enough) and fix $\xi_0 \in G(n,\sqrt{u},u^2)$. Define $\Lambda = \{x \in B_0(\sqrt{u}): \xi_0(x) = 1\}$.
By assumption, $\#\Lambda > n$. Also define the following configurations:
$$\xi_0' = \xi_0 \cdot \mathds{1}_{B_0(u^2)}, \qquad \xi_0^{\underline{1}}(x) = 1 \; \forall x\in \Z^d. $$
We consider the three processes $(\xi_t)$, $(\xi'_t)$ and $(\xi_t^{\underline{1}})$, respectively started from $\xi_0$, $\xi'_0$ and $\xi_0^{\underline{1}}$, constructed using the same augmented Harris system $\mathbb{H}$. Note that type 2 is absent from $(\xi'_t)$ and $(\xi_t^{\underline{1}})$, so that these are in fact one-type contact processes satisfying
\begin{align}\nonumber&\xi'_t(x) = 1 \quad \text{if and only if} \quad \exists \text{$y \in B_0(u^2)$ with $\xi'_0(y) = 1$} \\[-.6cm] \label{eq:SIP_propp1} \\ \nonumber&\hspace{5cm} \text{and there is an SIP from $(y,0)$ to $(x,t)$};\\[.2cm]
&\xi^{\underline{1}}_t(x) = 1 \quad \text{if and only if} \quad \text{there is an SIP from $\Z^d \times \{0\}$ to $(x,t)$}.\label{eq:SIP_propp2}
\end{align}
Also note that ($\xi_t^{\underline{1}}$) converges to $\mu_1$ as $t \to \infty$, so if $u$ is large enough, 
$$\left|\mathbb{E}[f(\xi^{\underline{1}}_u)] - \int fd\mu_1\right| < \varepsilon.$$ 
The statement of the lemma will thus follow once we prove that, if $u$ is large enough,
\begin{align}
\label{eq:aux_claim_1}& \P\left[\xi_u(x) = \xi'_u(x) \text{ for all } x \in B_0(\sqrt{u}) \right] > 1-\varepsilon \text{ and }\\[.2cm]
\label{eq:aux_claim_2} &\P\left[\xi'_u(x) = \xi^{\underline{1}}_u(x) \text{ for all } x \in B_0(\sqrt{u}) \right] > 1-\varepsilon.
\end{align}
The proof of \eqref{eq:aux_claim_1} is simple and we only sketch it. Observe that the process
$$\{x: \xi_t(x) \neq \xi'_t(x)\},\; t\geq 0$$
can be stochastically dominated by a (one-type) contact process with rate $\lambda_1$; this process is empty on $B_0(u^2)$ at time 0. Hence, \eqref{eq:aux_claim_1} follows from an application of \eqref{eq:1cp_moves_slowly}: from time 0 to time $u$, the occupied sites in this process do not have time to reach $B_0(\sqrt{u})$.

Let us prove \eqref{eq:aux_claim_2}. By \eqref{eq:SIP_propp1} and \eqref{eq:SIP_propp2}, we have $\xi'_u \leq \xi^{\underline{1}}_u$, so these two configurations can only differ in $B_0(\sqrt{u})$ if for some $x \in B_0(\sqrt{u})$ we have $\xi'_u(x) = 0$ and $\xi^{\underline{1}}_u(x) = 1$. Moreover, we have
\begin{equation}\label{eq:almost_over_bound}
\{\exists x \in B_0(\sqrt{u}):\xi'_u(x) =0,\; \xi^{\underline{1}}_u(x)= 1\} \subset E_1 \cup E_2,
\end{equation}
where
\begin{align*} &E_1 =\{\text{there is no SIP from $\Lambda \times \{0\}$ to $\Z^d \times \{t\}$}\},\\
&E_2 = \left\{\begin{array}{ll}\exists y \in \Lambda,\; x\in B_0(\sqrt{u}):& \text{there is no SIP from $(y,0)$ to $(x,u)$},\\[.2cm]& \text{there is an SIP from $(y,0)$ to $\Z^d \times \{u\}$},\\[.2cm]& \text{there is an SIP from  $\Z^d \times \{0\}$ to $(x,u)$}\end{array} \right\}.
\end{align*}
Hence, \eqref{eq:aux_claim_2} follows from \eqref{eq:almost_over_bound}, \eqref{eq:dies_many}, \eqref{eq:no_cross}, and a union bound. 
\end{proof}

\begin{proof}[Proof of Theorem \ref{thm:main2}]
Let $f: \{0,1,2\}^{\Z^d} \to \mathbb{R}$ be a function depending only on finitely many coordinates and fix $\xi_0 \in \{0,1,2\}^{\Z^d}$. By the Dominated Convergence Theorem, \begin{equation*}\mathbb{E}[f(\xi_t)\cdot \mathds{1}_{(\mathcal{S}_1 \cup \mathcal{S}_2)^c}] \xrightarrow{t \to \infty} f(\underline{0}) \cdot \P((\mathcal{S}_1 \cup \mathcal{S}_2)^c). \end{equation*}
We will prove that
\begin{align}\label{eq:dominate_part2}
&\mathbb{E}[f(\xi_t)\cdot \mathds{1}_{\mathcal{S}_1 }] \xrightarrow{t \to \infty} \int fd\mu_1 \cdot \P(\mathcal{S}_1) \text{ and }\\
&\mathbb{E}[f(\xi_t)\cdot \mathds{1}_{\mathcal{S}_1^c \cap \mathcal{S}_2 }] \xrightarrow{t \to \infty} \int fd\mu_2 \cdot \P(\mathcal{S}_1^c \cap \mathcal{S}_2)\label{eq:dominate_part3}
\end{align}
also hold. These three convergences imply in \eqref{eq:thm_main2}. The fact that the set of extremal stationary distributions is equal to $\{\mu_1,\mu_2,\delta_{\underline{0}}\}$ is an immediate consequence.

To prove \eqref{eq:dominate_part2}, assume $\xi_0$ has at least one site in state 1 and fix $\varepsilon > 0$. We choose variables as follows:
\begin{itemize}
\item choose $n_0, u_0$ corresponding to $f,\varepsilon$ in Lemma \ref{lem:Lchange4};
\item fix $n \geq n_0$ large enough corresponding to $\varepsilon$ in Lemma \ref{lem:many_survival_goes};
\item choose $r_0,s_0$ corresponding to $\xi_0,\varepsilon,n$ in Lemma \ref{lem:Lchange3};
\item fix $u \geq \max\{u_0,r_0^2\}$, then fix $s_1 \geq s_0$ with $\bbe_2 s_1 \geq u^2$, so that, by \eqref{eq:comparisonn},
\begin{equation*}
t \geq s_1 \quad \Longrightarrow \quad G(n,r_0,\bbe_2 t) \subset G(n, \sqrt{u},u^2).
\end{equation*}
\end{itemize}
With these choices, the implications of the three lemmas (Lemma \ref{lem:many_survival_goes}, \ref{lem:Lchange3} and \ref{lem:Lchange4}) give:
\begin{align}\label{eq:last_reason1}
&\P[\mathcal{S}_1 \mid \xi_t \in G(n,\sqrt{u},u^2)] > 1-\varepsilon \quad \forall t \geq 0,\\[.2cm]
\label{eq:last_reason2}&\P\left[\xi_t \in G(n,\sqrt{u},u^2)\mid \mathcal{S}_1\right] > 1-\varepsilon \quad \forall t \geq s_1,\\[.2cm]
\label{eq:last_reason3}&\left| \E \left[f(\xi_{t+u})\mid \xi_t \in G(n,\sqrt{u},u^2) \right] - \int fd\mu_1\right| < \varepsilon \quad \forall t \geq 0.
\end{align}

Now, for any $t \geq s_1 + u$ we have
\begin{align*}
\E[f(\xi_t) \cdot \mathds{1}_{\mathcal{S}_1}] - \int f d\mu_1 \cdot \P(\mathcal{S}_1) 
=& \E\left[f(\xi_t)\cdot \left(\mathds{1}_{\mathcal{S}_1} - \mathds{1}_{\{\xi_{t-u} \in G(n,\sqrt{u},u^2)\}} \right)\right]\\&+  \mathbb{E}\left[\left(f(\xi_t) - \int f d\mu_1 \right) \cdot \mathds{1}_{\{ \xi_{t-u} \in G(n,\sqrt{u},u^2)\}} \right] \\&+ \int f d\mu_1 \cdot \left(\P[\xi_{t-u} \in G(n,\sqrt{u},u^2)] - \P(\mathcal{S}_1)\right).
\end{align*}
We bound the absolute values of the three terms on the right-hand side as follows. By \eqref{eq:last_reason1} and \eqref{eq:last_reason2},
$$\left| \E\left[f(\xi_t)\cdot \left(\mathds{1}_{\mathcal{S}_1} - \mathds{1}_{\{\xi_{t-u} \in G(n,\sqrt{u},u^2)\}} \right)\right]\right|\leq \|f\|_\infty \cdot (\P(\mathcal{S}_1) - \P\left[\xi_{t-u} \in G(n,\sqrt{u},u^2) \right]< 2\varepsilon  \|f\|_\infty  $$
and
$$ \left|\int f d\mu_1 \cdot \left(\P[\xi_{t-u} \in G(n,\sqrt{u},u^2)] - \P(\mathcal{S}_1)\right)\right| \leq  2\varepsilon \|f\|_\infty;$$
next, by \eqref{eq:last_reason3},
$$\left| \mathbb{E}\left[\left(f(\xi_t) - \int f d\mu_1 \right) \cdot \mathds{1}_{\{ \xi_{t-u} \in G(n,\sqrt{u},u^2)\}} \right] \right| \leq \varepsilon. $$
This proves that, for any $t \geq s_1 + u$,
$$ \left| \E[f(\xi_t) \cdot \mathds{1}_{\mathcal{S}_1}] - \int f d\mu_1 \cdot \P(\mathcal{S}_1) \right| < \varepsilon + 4\varepsilon \|f\|_\infty,$$
proving \eqref{eq:dominate_part2}.

Let us now prove \eqref{eq:dominate_part3}. If $0 < s < t$, we have
\begin{align*}
\E\left[f(\xi_t) \cdot \mathds{1}_{\mathcal{S}_1^c \cap \mathcal{S}_2}\right] - \int f d\mu_2 \cdot \P(\mathcal{S}_1^c \cap \mathcal{S}_2) =& \mathbb{E}\left[f(\xi_t) \cdot \mathds{1}_{\mathcal{S}_2} \cdot \left(\mathds{1}_{\mathcal{S}_1^c} - \mathds{1}_{\{\xi_s \in \mathscr{A}_2\}} \right) \right]\\[.2cm]
&+ \mathbb{E}\left[\left(f(\xi_t) - \int fd\mu_2\right) \cdot \mathds{1}_{\{\xi_s \in \mathscr{A}_2\}\cap \mathcal{S}_2}   \right] \\[.2cm]
&+ \int f d\mu_2 \left(\P(\{\xi_s \in \mathscr{A}_2\}\cap \mathcal{S}_2) - \P(\mathcal{S}_1^c \cap \mathcal{S}_2 )\right) .
\end{align*}
Since $\mathcal{S}_1^c = \cup_{s\geq 0} \{\xi_s \in \mathscr{A}_2\}$, the first and third terms on the right-hand side can be made arbitrarily small if $s$ is large enough. Next, the second term on the right-hand side is equal to
$$\E\left[\mathds{1}_{\{\xi_s \in \mathscr{A}_2\}} \cdot \left( \E\left[\left.f(\xi_t) \cdot \mathds{1}_{\mathcal{S}_2}\right| \xi_s \right] - \int fd\mu_2 \cdot \P\left[\left. \mathcal{S}_2\right| \xi_s\right]\right) \right] \xrightarrow{t \to \infty} 0,$$
by the complete convergence theorem for the one-type contact process.
\end{proof}

\section{Reversing time, steering paths}
\label{s:ancestor}
So far we have proved our main results assuming the validity of Proposition \ref{prop:cone0}. Proving this proposition will be the focus of our efforts in the remainder of the paper. In this section, we perform three tasks:\begin{itemize}\item First, we state a modified version of Proposition \ref{prop:cone0} (see Proposition \ref{prop:cone0_dual} below) which is more convenient to prove.\item Second, we state a result (Proposition \ref{prop:steer} below) which is our essential tool in proving  Proposition \ref{prop:cone0_dual}. The proof of Proposition \ref{prop:steer} is postponed to Section \ref{s:times} and the Appendix.\item Third, we show how Proposition \ref{prop:steer} implies Proposition \ref{prop:cone0_dual} (though part of this argument is again postponed to the Appendix).\end{itemize}

In what follows, we will often refer to time restrictions and space-time shifts of augmented Harris systems; let us introduce these. Given an augmented Harris system $\mathbb{H}$ and an interval $I \subseteq [0,\infty]$, the \textit{restriction of $\mathbb{H}$ to $I$} is the triple
$$\mathbb{H}_{I}= \left( (D^x_{I}), (D^{x,y}_{I}), (\mathscr{D}^{x,y}_{I}) \right) = \left( (D^x \cap I), (D^{x,y} \cap I), (\mathscr{D}^{x,y} \cap I)\right).$$
Let $\Omega$ be the set of all possible realizations of $\mathbb{H}$. %An \textit{event on Harris systems} is a measurable subset of
%$$ \{\mathbb{H}_I: \mathbb{H} \in \Omega,\; I \subset [0,\infty] \text{ interval}  \}. $$

Given $(x_0,t_0) \in \Z^d \times [0,\infty)$, we define the \textit{space-time shift of $\mathbb{H}$ by $(x_0,t_0)$} by
$$[\theta(x_0,t_0)](\mathbb{H}) = \left(([\theta(x_0,t_0)](D^x)),\; ([\theta(x_0,t_0)](D^{x,y})),\; ([\theta(x_0,t_0)](\mathscr{D}^{x,y})) \right),$$
where $[\theta(x_0,t_0)](D^x) = \{t-t_0:\; t \in D^{x_0+x} \cap [t_0,\infty)\}$, and similarly for $[\theta(x_0,t_0)](D^{x,y})$ and $[\theta(x_0,t_0)](\mathscr{D}^{x,y})$. If $X = X(\mathbb{H})$ is a function of augmented Harris systems, we denote  $[X \circ \theta(x_0,t_0)](\mathbb{H}) = X([\theta(x_0,t_0)](\mathbb{H}))$. In this notation, we will often omit $\mathbb{H}$ and simply write $X \circ \theta(x_0,t_0)$.

\subsection{Time reversal of Proposition \ref{prop:cone0}}
We start with some definitions. As in the previous section, we fix an augmented Harris system $\mathbb{H} = (H,\mathcal{H})$.
\begin{definition}
A \emph{reverse free basic infection path} (RFBIP) is a basic infection path $\gamma: [t_1,t_2]\to\Z^d$ satisfying
\begin{equation}
s \in [t_1,t_2],\;\gamma(s) \neq \gamma(s-) \quad \Longrightarrow \quad  (\gamma(s-), s+) \not \rsa \Z^d \times \{t_2\}.
\label{eq:prop_reverse}\end{equation}
A \emph{reverse free selective infection path} (RFSIP) is a selective infection path satisfying \eqref{eq:prop_reverse}.
\end{definition}
The reason for using the word `reverse' will be clear in a moment. 
\begin{definition} Define the space-time sets
\begin{align*}
&\mathscr{R}'(x,t,\ell) = [x-\ell,x+\ell]^d \times [t, t+\ell],\quad x \in \Z^d,\;t,\ell \geq 0.\\
&\mathscr{C}'(x,t,\alpha) = \{(y,s) \in \Z^d \times [0,t]:\;\|y-x\| \leq  \alpha(t-s)\},\quad x \in \Z^d,\; t, \alpha \geq 0.
\end{align*}
\end{definition}
We are now ready to state
\begin{proposition}
\label{prop:cone0_dual}
Assume $\lambda_1 > \lambda_2 > \lambda_c$. There exists $\bar{c} > 0$ and $\bbe > 0$ such that the following holds. For any $u > t > 0$, $\ell \in (0, u-t)$ and $x,y \in \Z^d$ with $(x,0) \in \mathscr{C}'(y,t,\bbe)$, we have
\begin{equation*}
\P\left[\left\{(x,0) \not \rsa \Z^d \times \{u\} \right\} \cup \left\{\begin{array}{r} \exists (y',t') \in \mathscr{R}'(y,t,\ell): \;(y',t') \rsa \Z^d \times \{u\} \\[.2cm]\text{and $\exists$ an RFSIP from $(x,0)$ to $(y',t')$} \end{array}\right\}\right] > 1-\exp(-\bar{c}\ell).
\end{equation*}
\end{proposition}
To show that this is equivalent to Proposition \ref{prop:cone0}, fix $u>0$ and consider $\mathbb{H}_{[0,u]}= \left( (D^x_{[0,u]}), (D^{(x,y)}_{[0,u]}), (\mathcal{D}^{(x,y)}_{[0,u]}) \right),$ the restriction of $\mathbb{H}$ to the time interval $[0,u]$. We now define $\mathbb{H}^*_{[0,u]}$ as the augmented Harris system on $[0,u]$ defined from $\mathbb{H}_{[0,u]}$ by reversing the sense of time and of the arrows. Formally, we let $$\mathbb{H}^*_{[0,u]}=\left((B^x_{[0,u]}), (B^{(x,y)}_{[0,u]}), (\mathcal{B}^{(x,y)}_{[0,u]}) \right),$$
where
\begin{align*}
&B^x_{[0,u]} = \{t \in [0,u]:\; u-t \in D^x\},\quad x \in \Z^d,\\[.2cm]
&B^{(x,y)}_{[0,u]} = \{t \in [0,u]:\; u-t \in D^{(y,x)}\},\quad x,y \in \Z^d, \; 0< \|x-y\| \leq R,\\[.2cm]
&\mathcal{B}^{(x,y)}_{[0,u]} = \{t \in [0,u]:\; u-t \in \mathscr{D}^{(y,x)}\},\quad x,y \in \Z^d, \; 0< \|x-y\| \leq R\end{align*}
respectively give the sets of death marks, arrows and selective arrows of $\mathbb{H}^*_{[0,u]}$. Given a function $\gamma: [t_1,t_2] \to \Z^d$ with $0\leq t_1 \leq t_2 \leq u$, define $\gamma^*: [u-t_2,u-t_1] \to \Z^d$ by setting $\gamma^*(t) = \gamma(u-t)$ for each $t$. Then, it is readily seen that $\gamma$ is respectively a BIP, SIP, RFBIP, or RFSIP with respect to $\mathbb{H}$ if and only if $\gamma^*$ is respectively a BIP, SIP, FBIP, or FSIP with respect to $\mathbb{H}^*_{[0,u]}$. This, together with the fact that $\mathbb{H}_{[0,u]}$ and $\mathbb{H}^*_{[0,u]}$ have the same distribution, implies the equivalence between Propositions \ref{prop:cone0} and \ref{prop:cone0_dual}.

It will be useful to note that, as a consequence of Lemma \ref{lem:unique_FBIP} (applied to $\mathbb{H}^*_{[0,u]}$) we have that, in $\mathbb{H}$,
\begin{equation}\label{eq:dual_unique_prop}
\begin{array}{ll}\forall x \in \Z^d,\; 0 \leq t_1 \leq t_2,&\text{either } (x,t_1) \not \rsa \Z^d \times \{t_2\} \\&\text{ or there is a unique RFBIP from $(x,t_1)$ to $\Z^d \times \{t_2\}$.}\end{array}
\end{equation}
In order to find the unique RFBIP mentioned in \eqref{eq:dual_unique_prop}, one can follow a procedure that is a time reversal of what was explained after Lemma \ref{lem:unique_FBIP}. Namely, start with an arbitrary BIP $\upgamma$ from $(x,t_1)$ to $\Z^d \times \{t_2\}$, consider the \textit{smallest} jump time $s$ for which \eqref{eq:prop_reverse} is violated in $\upgamma$, take another BIP $\hat{\upgamma}$ from $(\upgamma(s), s+)$ to $\Z^d \times \{t_2\}$, define $\upgamma_1 = \upgamma \cdot \mathds{1}_{[t_1,s]} + \hat{\upgamma} \cdot \mathds{1}_{(s,t_2]}$, and so on.

As a consequence of Lemma \ref{lem:concatenation}, we have that
\begin{equation}\label{eq:dual_conc_prop}
\begin{split}&\text{if }\gamma_1: [t_1,t_2] \to \Z^d,\; \gamma_2: [t_2,t_3] \to \Z^d \text{ are RFSIP's with $\gamma_1(t_2) = \gamma_2(t_2)$, then }\\&  \gamma = \gamma_1 \mathds{1}_{[t_1,t_2]} + \gamma_2 \mathds{1}_{(t_2,t_3]}\text{ is an RFSIP; if $\gamma_1,\gamma_2$ are RFBIP's, then $\gamma$ is an RFBIP}.
\end{split}
\end{equation}
It is often fruitful to consider the two systems $\mathbb{H}_{[0,u]}$ and $\mathbb{H}^*_{[0,u]}$ jointly and exploit duality-type relations between them. However, we will not need to do so in the rest of the paper. From now on, we will have a single augmented Harris system $\mathbb{H}$ (defined on $[0,\infty)$) and will  work on proving that the set of BIP's, SIP's, RFBIP's and RFSIP's of $\mathbb{H}$ are such that Proposition \ref{prop:cone0_dual} is satisfied. In particular, we will use properties \eqref{eq:dual_unique_prop} and \eqref{eq:dual_conc_prop} without making reference to a time-reversed copy of the augmented Harris system.

\subsection{Steering reverse free selective infection paths}
The essential tool in our proof of Proposition \ref{prop:cone0_dual} will be the following. We denote by $e_1,\ldots, e_d$ the canonical vectors of $\mathbb{Z}^d$.
\begin{proposition}
\label{prop:steer}
Assume $\lambda_1 > \lambda_2 > \lambda_c$. On the event $\{(0,0) \rightsquigarrow \infty\}$, there exist random variables $\mathcal{T} \in [0,\infty)$ and $\mathcal{X} \in \mathbb{Z}^d$ such that
\begin{itemize}
\item[(1)] there is an RFSIP from $(0,0)$ to $(\mathcal{X}, \mathcal{T})$;
\item[(2)] for any events $E_1$ and $E_2$ on augmented Harris systems,
\begin{equation}\begin{split}
&\P\left[\mathbb{H}_{[0,\mathcal{T}]} \in E_1,\; \mathbb{H}\circ \theta(\mathcal{X},\mathcal{T}) \in E_2 \mid (0,0) \rsa \infty\right] \\&\hspace{2cm}= {\P}\left[\mathbb{H}_{[0,\mathcal{T}]} \in E_1\mid (0,0) \rsa \infty\right] \cdot {\P}\left[\mathbb{H} \in E_2\mid (0,0) \rsa \infty\right];\end{split}
\end{equation}
\item[(3)] if $\sigma > 0$ is small enough, $${{\mathbb{E}}[\exp(\sigma \mathcal{T})\mid (0,0) \rsa \infty] < \infty,\;\; {\mathbb{E}}[\exp(\sigma \cdot \|\mathcal{X}\|)\mid (0,0) \rsa \infty] < \infty;}$$
\item[(4)] ${\mathbb{E}}[\mathcal{X}\mid (0,0) \rsa \infty] = \alpha e_1,$
where $\alpha > 0$.
\end{itemize}
\end{proposition}
The proof of Proposition \ref{prop:steer} will be carried out in the next section and the Appendix. In the remainder of this section, we show how it is employed to prove Proposition \ref{prop:cone0_dual}.

\begin{definition} Given $i\in\{1,\ldots, d\}$ and $\kappa \in \{-1,1\}$, on the event $\{(0,0) \rsa \infty\}$ we define $(\mathcal{X}^{i,\kappa},\mathcal{T}^{i,\kappa}) \in \Z^d \times [0,\infty)$ by
$$(\mathcal{X}^{i,\kappa}(\mathbb{H}),\mathcal{T}^{i,\kappa}(\mathbb{H}))= (\mathcal{X}(\psi^{(i,\kappa)}(\mathbb{H})),\mathcal{T}(\psi^{(i,\kappa)}(\mathbb{H}))),$$
where $\psi^{(i,\kappa)} : \R^d \to \R^d$ is the linear transformation given by 
\begin{equation}\label{eq:linear_psi}\psi^{(i,\kappa)}(e_1) = \kappa e_i,\quad \psi^{(i,\kappa)}(e_i) = \kappa e_1,\quad \psi^{(i,\kappa)}(e_j) = e_j \text{ for all }j \notin \{1,i\}.\end{equation}\end{definition} 
Note that $(\mathcal{X}^{(1,1)},\mathcal{T}^{(1,1)}) = (\mathcal{X},\mathcal{T})$. Since $\mathbb{H}$ and $\psi^{(i,\kappa)}(\mathbb{H})$ have the same distribution, $(\mathcal{X}^{(i,\kappa)}, \mathcal{T}^{(i,\kappa)})$ satisfies properties (1)-(3) of Proposition \ref{prop:steer}, and property (4) is replaced by
\begin{equation}\label{eq:kappa_e_i}
{\mathbb{E}}[\mathcal{X}^{(i,\kappa)}\mid (0,0) \rsa \infty] =  \kappa \alpha e_i.
\end{equation}
Moreover, the distributions of $\mathcal{T}$ and $\mathcal{T}^{(i,\kappa)}$ are the same.
\begin{definition}
Let $\vec{\kappa} \in \{-1,1\}^d$. On the event $\{(0,0) \rsa\infty\}$, we define a random space-time point $(\mathcal{X}^{\vec{\kappa}},\mathcal{T}^{\vec{\kappa}})$ as follows. First define the following vectors recursively:
$$(x_1,t_1) = (\mathcal{X}^{1,\kappa_1},\mathcal{T}^{1,\kappa_1}),\quad  (x_i,t_i) = (x_{i-1},t_{i-1}) + (\mathcal{X}^{i,\kappa_i},\mathcal{T}^{i,\kappa_i}) \circ \theta(x_{i-1},t_{i-1}),\;\; i \in \{2,\ldots, d\}.$$
Then, put $(\mathcal{X}^{\vec{\kappa}},\mathcal{T}^{\vec{\kappa}}) = (x_d,t_d)$. \end{definition} The following is then an immediate consequence of Proposition \ref{prop:steer} and the identity \eqref{eq:kappa_e_i}.
\begin{corollary}\label{cor:steer}
Assume $\lambda_1 > \lambda_2 > \lambda_c$.  For any $\vec{\kappa} = (\kappa_1,\ldots, \kappa_d) \in \{-1,1\}^{d}$,
\begin{itemize}
\item[(1)] there is an RFSIP from $(0,0)$ to $(\mathcal{X}^{\vec{\kappa}}, \mathcal{T}^{\vec{\kappa}})$;
\item[(2)] for any events $E_1$ and $E_2$ on augmented Harris systems,
\begin{equation}\begin{split}
&\P\left[\mathbb{H}_{[0,\mathcal{T}^{\vec{\kappa}}]} \in E_1,\; \mathbb{H}\circ \theta(\mathcal{X}^{\vec{\kappa}},\mathcal{T}^{\vec{\kappa}}) \in E_2 \mid (0,0) \rsa \infty\right] \\&\hspace{2cm}= {\P}\left[\mathbb{H}_{[0,\mathcal{T}^{\vec{\kappa}}]} \in E_1\mid (0,0) \rsa \infty\right] \cdot {\P}\left[\mathbb{H} \in E_2\mid (0,0) \rsa \infty\right];\end{split}
\end{equation}
\item[(3)] {$\displaystyle{{\mathbb{E}}[\exp(\sigma \mathcal{T}^{\vec{\kappa}})\mid (0,0) \rsa \infty] < \infty,\;\; {\mathbb{E}}[\exp(\sigma \cdot \|\mathcal{X}^{\vec{\kappa}}\|)\mid (0,0) \rsa \infty] < \infty;}$}
\item[(4)] ${\mathbb{E}}[\mathcal{X}^{\vec{\kappa}}\mid (0,0) \rsa \infty] = \alpha \sum_{i=1}^d\kappa_i\cdot e_i;$
\item[(5)] the distribution of $$\left(|\mathcal{X}^{\vec{\kappa}} \cdot e_1|, \ldots, |\mathcal{X}^{\vec{\kappa}} \cdot e_d|, \mathcal{T}^{\vec{\kappa}}\right)$$ does not depend on $\vec{\kappa}$.
\end{itemize}
\end{corollary}
\begin{definition} Given $x \in \Z^d$, on the event $\{(x,0) \rsa \infty\}$, define a sequence $(S_n,\uptau_n)$ as follows. Let $(S_0,\uptau_0) = (x,0)$ and, recursively,
$$(S_{n+1},\uptau_{n+1}) = (S_n, \uptau_n) + (\mathcal{X}^{\vec{\kappa}^{(n)}},\mathcal{T}^{\vec{\kappa}^{(n)}}) \circ \theta(S_n,\uptau_n), \qquad n \geq 0,$$
where $\vec{\kappa}^{(n)} \in \{-1,1\}^d$ is defined by
$$\vec{\kappa}^{(n)} \cdot e_i = \begin{cases} 1&\text{if } S_n \cdot e_i\leq 0;\\-1&\text{if } S_n \cdot e_i > 0,\end{cases} \qquad i = 1,\ldots, d.$$ \end{definition}
By Corollary \ref{cor:steer}, $(\uptau_n)_{n \geq 0}$ is a renewal sequence, and $(S_n)_{n\geq 0}$ is a Markov chain on $\Z^d$ such that, in each coordinate, from outside the origin, the step distribution has a drift in the direction of the origin. We will need two properties of $(S_n,\uptau_n)$ in what follows. First, for each $n$ there is an RFSIP from $(S_n,\uptau_n)$ to $(S_{n+1},\uptau_{n+1})$; hence, concatenating as in \eqref{eq:dual_conc_prop}, 
\begin{equation}\label{eq:concatenate_serial} \text{for each $n$ there exists an RFSIP $\upgamma:[0,\uptau_n] \to \Z^d$  with $\upgamma(0) = x$ and $\upgamma(\uptau_n) = S_n$.}\end{equation}
Second, by part (2) of Corollary \ref{cor:steer}, for each $n$, under $\mathbb{P}(\cdot \mid (x,0) \rsa \infty)$, the distribution of $\mathbb{H} \circ (S_n,\uptau_n)$ is equal to $\mathbb{P}(\cdot \mid (0,0) \rsa \infty)$.  Consequently, \begin{equation} \label{eq:join_serial}\P[(S_n,\uptau_n) \rsa \infty \;\forall n \mid (x,0) \rsa \infty] = 1.\end{equation}

The following is a tightness-type result for the sequence $(S_n, \uptau_n)$.
\begin{lemma}\label{lem:cone_finally}
Assume $\lambda_1 > \lambda_2 > \lambda_c$. There exists $\bar{c}_0 > 0$ and $\bbe > 0$ such that the following holds. For any $u > t > 0$, $\ell \in (0, u-t)$ and $x,y \in \Z^d$ with $(x,0) \in \mathscr{C}'(y,t,\bbe)$, if $(S_0,\uptau_0) = (x,0)$,
\begin{equation}\label{eq:good_n_steer}
\P\left[\exists n: (S_n,\uptau_n) \in \mathscr{R}'(y,t,\ell)\mid (x,0) \rsa \infty\right] > 1-\exp(-\bar{c}_0\ell).
\end{equation}
\end{lemma}
Since this result is more about random walks embedded in renewal times than it is about the multitype contact process, we deal with it in the Appendix. (Lemma \ref{lem:cone_finally} follows from Proposition \ref{prop:cone} in the Appendix. Note that Proposition \ref{prop:cone} assumes that the spatial coordinate is one-dimensional; in order to obtain Lemma \ref{lem:cone_finally}, we must apply Proposition \ref{prop:cone} to $(S_n \cdot e_i, \uptau_n)_{n\geq 0}$ for each $i$, together with a union bound).

\begin{proof}[Proof of Proposition \ref{prop:cone0_dual}] Fix $t,u,\ell,x,y$ as in the statement of the proposition. It suffices to prove that there exist $c > 0$ and $\bbe > 0$ such that
\begin{equation}\label{eq:the_prob}
\P\left[\left.\begin{array}{r}\exists (y',t') \in \mathscr{R}'(y,t,\ell): (y',t') \rsa \Z^d \times \{u\}\\[.2cm] \text{ and  $\exists$ an RFSIP from $(x,0)$ to $(y',t')$} \end{array} \right| (x,0) \rsa \Z^d \times \{u\}\right] > 1 - \exp(-c\ell).
\end{equation}
Let $E$ be the event inside the conditional probability. We start by bounding:
\begin{align}\nonumber
\P[E\mid(x,0) \rsa \Z^d \times \{u\}] &\geq \P[(x,0) \rsa \Z^d \times \{u\}]^{-1}\cdot \P\left[E \cap \{(x,0) \rsa \infty\}\right]\\[.2cm]
&= \frac{\P[(x,0) \rsa \infty]}{\P[(x,0) \rsa \Z^d \times \{u\}]} \cdot \P\left[E \mid (x,0) \rsa \infty\right]. \label{eq:refer_terms}
\end{align}
Next, we bound
\begin{align*}
\frac{\P[(x,0) \rsa \infty]}{\P[(x,0) \rsa \Z^d \times \{u\}]} = 1- \frac{\P[(x,0) \rsa \Z^d \times \{u\},\; (x,0) \not \rsa \infty]}{ \P[(x,0) \rsa \Z^d \times \{u\}]} \stackrel{\eqref{eq:dies_quickly}}{\geq} 1- \frac{\exp(-c\cdot u)}{\P[(0,0) \rsa \infty]}.
\end{align*}

Let us treat the second term in \eqref{eq:refer_terms}. Assume that there exists $n$ as in the conditional probability in \eqref{eq:good_n_steer}; let us show that $E$ then occurs. As noted in \eqref{eq:concatenate_serial}, there is an RFSIP from $(x,0)$ to $(S_n, T_n)$. Then, by \eqref{eq:good_n_steer}, we have $(S_n,T_n) \rsa \infty$, so $(S_n,T_n) \rsa \Z^d \times \{u\}$, so by \eqref{eq:dual_unique_prop} there is a unique RFBIP from $(S_n,T_n)$ to $\Z^d \times \{u\}$. The conclusion follows by concatenating the RFSIP and the RFBIP, as in \eqref{eq:dual_conc_prop}.
\end{proof}

\section{Ancestor process and renewal-type random times}
\label{s:times}

In this section we prove Proposition \ref{prop:steer}, the building block of our steering procedure. \textit{Throughout this section, we assume that $\lambda_1 > \lambda_2 > \lambda_c$,} since this condition is assumed in Proposition \ref{prop:steer}. We start defining an auxiliary process, first introduced and studied in \cite{neuhauser}, which will be a useful tool in our proofs.

\subsection{Ancestor process}
\begin{definition}
Given $(x,s) \in \Z^d \times [0,\infty)$, the \emph{ancestor process} of $(x,s)$, denoted $(\eta^{(x,s)}_t)_{t \geq s}$, is the process taking values on $\Z^d \cup \{\triangle\}$ defined as follows. For each $t \geq s$,
\begin{itemize}
\item if $(x,s) \not \rsa \Z^d \times \{t\}$, let $\eta^{(x,s)}_t = \triangle$;
\item if $(x,s) \rsa \Z^d \times \{t\}$, let $\eta^{(x,s)}_t = \gamma(t)$, where $\gamma$ is the unique RFBIP $\gamma: [s,t] \to \Z^d$ with $\gamma(s) = x$ (see \eqref{eq:dual_unique_prop}). 
\end{itemize}
In case $(x,s) = (0,0)$, we write $\eta_t$ instead of $\eta^{(0,0)}_t$.
\end{definition}
This process has been introduced in \cite{neuhauser} and further studied in \cite{valesin}.
 We emphasize that $(\eta^{(x,s)}_t)_{t \geq 0}$ only depends on $\mathbb{H} = (H,\mathcal{H})$ through $H$. Also note that for each $t \geq s$, $\eta^{(x,s)}_t$ only depends on $H_{[s,t]}$.  A useful consequence of \eqref{eq:dual_conc_prop} is:
\begin{equation}
0\leq s<t<u,\; x,y,z\in\Z^d,\; \eta^{(x,s)}_t = y,\; \eta^{(y,t)}_u = z \quad \Longrightarrow \quad \eta^{(x,s)}_u = z.
\end{equation}

Let us clarify one potential source of confusion in the definition of the ancestor process. If $\eta^{(x,s)}_t = y \neq \triangle$, then by definition there is an RFBIP $\gamma: [s,t] \to \Z^d$ with $\gamma(s) = x$ and $\gamma(t) = y$, but $\gamma$ does not necessarily coincide  with the path $[s,t] \ni r \mapsto \eta^{(x,s)}_r$. In fact, $r \mapsto \eta^{(x,s)}_r$ needs not even be a basic infection path. See Figure \ref{fig:illustrate} for an example of the ancestor process which illustrates this distinction.

\begin{figure}[htb]
\begin{center}
\setlength\fboxsep{0pt}
\setlength\fboxrule{0.5pt}
\fbox{\includegraphics[width = 0.8\textwidth]{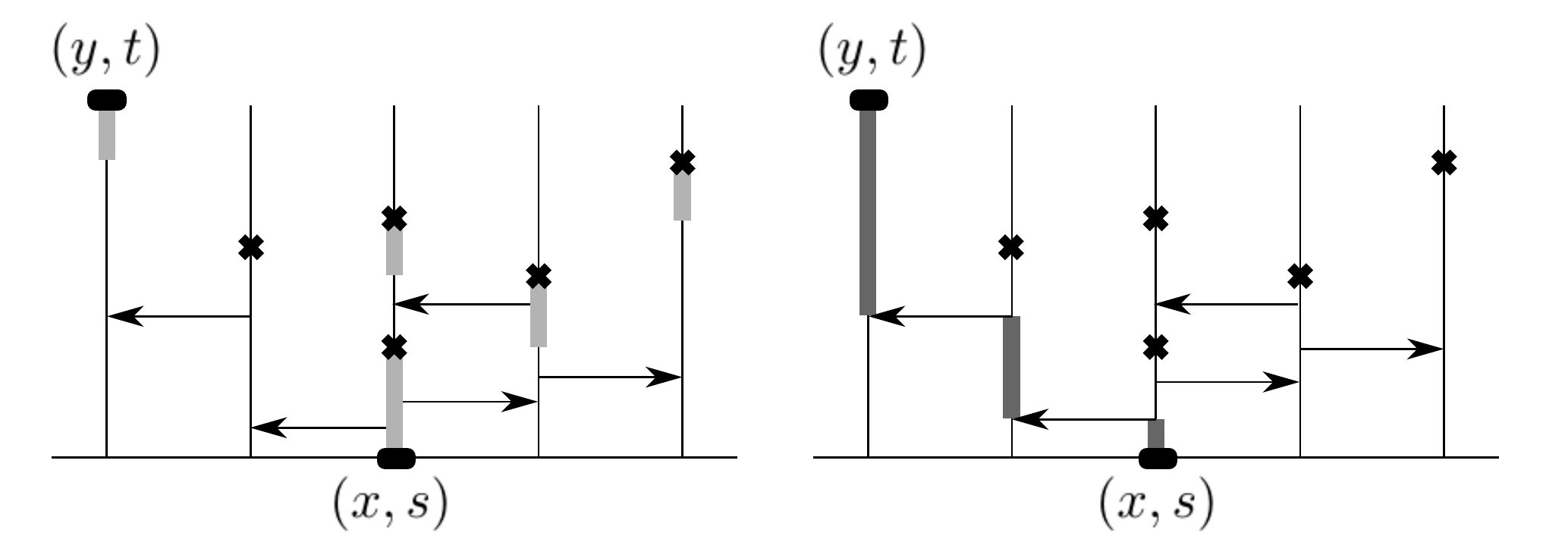}}
\end{center}
\caption{Left: the thick gray trajectory represents the values of $\eta^{(x,s)}_r$ for each $r \in [s,t]$ (in order to identify this process, one can use the procedure mentioned after \eqref{eq:dual_unique_prop}). Right: the  thick (darker) gray trajectory shows the unique RFBIP from $(x,s)$ to $\Z \times \{t\}$.}
\label{fig:illustrate}
\end{figure}

\iffalse From this it also follows that
\begin{equation}
0 \leq s < t,\; x,y\in \Z^d,\; \eta^{(x,s)}_t = y,\; (y,t)\rsa \infty \quad \Longrightarrow \quad \eta^{(x,s)}_u = \eta^{(y,t)}_u \;\forall u\geq t.
\end{equation}
\fi

\subsection{Introducing a drift: proof of Proposition \ref{prop:steer}}
Our proof of Proposition \ref{prop:steer} consists of two ``ingredients'', each involving the definition of a random space-time point and the discussion of some of its properties. These ingredients are then combined to define the space-time point $(\mathcal{X},\mathcal{T})$ of Proposition \ref{prop:steer}.
\subsubsection{Ingredient 1: Bifurcation times of the ancestor process}
We start our steering construction defining random times at which the ancestor process of $(0,0)$, $(\eta_t)_{t\geq 0}$, satisfies a list of conditions. The construction depends on a constant $L \in \mathbb{N}$, $L > 1$, which we will choose later.

In what follows, the word `arrow' does \textit{not} refer to selective arrows. 
Let $t \geq 3$ and assume that $(0,0) \rsa \Z^d \times \{t\}$, so that $\eta_t \neq \triangle$.  Let $y = \eta_{t-3}$. We say that $t$ is a \textit{bifurcation time} of the ancestor process $\eta$ if we can find:
\begin{itemize}
\item no death mark in $\{y-e_1,y,y+e_1\} \times [t-3,t-1]$;
\item a death mark in $\{y\} \times [t-1,t]$;
\item exactly two arrows started from $\{y\} \times [t-3,t-2]$; one from $y$ to $y + e_1$ and the other from $y$ to $y - e_1$;
\item no arrow started from $\{y\} \times [t-2,t]$;
\item no arrow started from $\{y-e_1,y+e_1\} \times [t-3,t-1]$;
\item exactly one basic infection path from $(y-e_1,t-1)$ to $\Z^d \times \{t\}$, ending at $y-Le_1$;
\item exactly one basic infection path from $(y+e_1,t-1)$ to $\Z^d \times \{t\}$, ending at $y+Le_1$.
\end{itemize}
Note that  bifurcation times depend on $H$ and not on $\mathcal{H}$. See Figure \ref{fig:bifurcation} for an illustration of a bifurcation time.

\begin{figure}[htb]
\begin{center}
\setlength\fboxsep{0pt}
\setlength\fboxrule{0.5pt}
\fbox{\includegraphics[width = 0.7\textwidth]{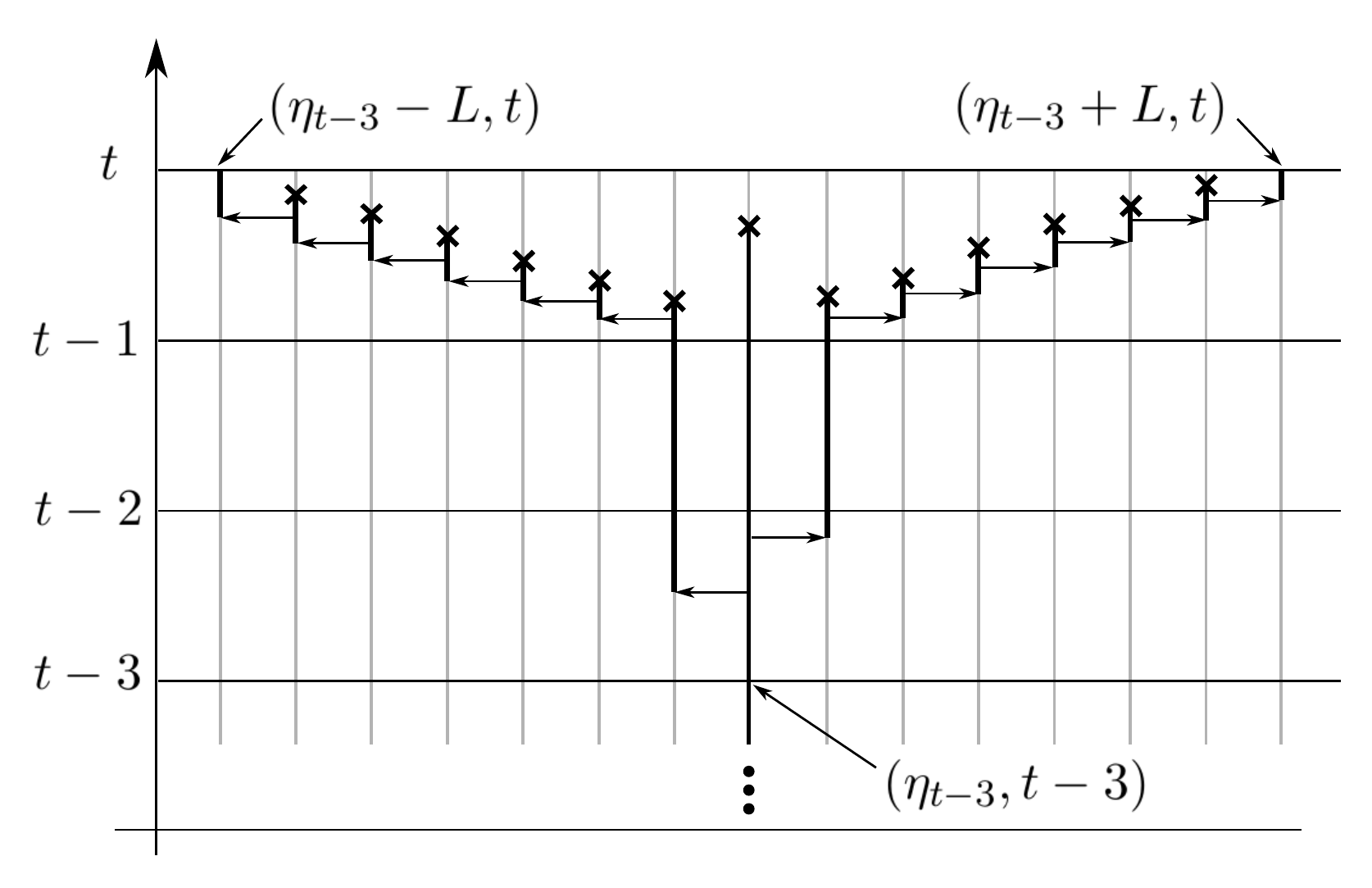}}
\end{center}
\caption{Illustration of a bifurcation time $t$ of the ancestor process $\eta$ (in the $d = 1$ case).}
\label{fig:bifurcation}
\end{figure}

\iffalse\begin{remark} Assume that $t$ is a bifurcation time and that, of the two arrows started from $\{y\} \times [t-3,t-2]$, the one ending at $z-e_1$ comes earlier than the one ending at $z+e_1$. Then, there is an $H$-good infection path from $(z,t-3)$ to $(z-L,t)$, so that $$\eta_t = \eta^{(z,t-3)}_t = z-L.$$ If, on the other hand, the arrow to $z+1$ comes first, then we have $$\eta_t = \eta^{(z,t-3)}_t = z+L.$$\end{remark}\fi

We now want to find a bifurcation time $t$ around a spatial location $y$ with the property that $(y-Le_1,t) \rsa \infty$ or $(y+Le_1,t) \rsa \infty$ (or both). Let us first give an heuristic explanation to our approach to find such a point. We start following the ancestor process $\eta$ until  a bifurcation time $u$ is found (the bifurcation occurs in the time interval $[u-3,u]$, around a spatial location $y = \eta_{u-3}$). We then ask if at least one of $(y-Le_1,u) \rsa \infty$ and $(y+Le_1,u) \rsa \infty$ holds. If the answer is affirmative, we are done with our search. Otherwise, we wait until the first time $v$ such that $\{(y-Le_1,u),(y+Le_1,u)\} \not \rsa \Z^d \times \{v\}$; then we look for a new bifurcation time $u'$ after $v$, and repeat the procedure. 

We now give the rigorous description of this procedure. Let $U_1$ be the smallest  bifurcation time in $H$ (with the convention that $U_1 = \infty$ if there is no bifurcation time). Note that $U_1$ is a stopping time with respect to the natural filtration of the Poisson processes in $H$. In the event $\{U_1 < \infty\}$, we let $Y_1 = \eta_{U_1 - 3}$, that is, $Y_1$ is the spatial position around which the bifurcation at $U_1$ occurs. Then define another stopping time $V_1$ as: 
$$V_1 = \begin{cases}\sup\{t: \{(Y_1 - Le_1, U_1), (Y_1 + Le_1, U_1)\} \rsa \Z^d \times \{t\}\}&\text{if } U_1 < \infty;\\\infty&\text{otherwise}. \end{cases}$$
In words, in case $U_1 < \infty$, $V_1$ is the supremum of all times $t$ that can be reached by BIP's $\gamma:[U_1, t] \to \Z^d$ with $\gamma(U_1) \in \{Y_1 - Le_1, Y_1 + Le_1\}$.
Next, if $V_1 < \infty$, $\eta_{V_1} \neq \triangle$ and the ancestor process has at least one bifurcation time $t \geq V_1 + 3$, we let $U_2$ be the smallest bifurcation time larger than $V_1 + 3$, and let $Y_2 = \eta_{U_2-3}$. In all other cases (that is, (a) if $V_1 = \infty$, (b) if $V_1 < \infty$ but $\eta_{V_1} = \triangle$, or (c) if the ancestor process has no bifurcation time after $V_1 + 3$), we let $U_2 = \infty$ and $Y_2 = \triangle$. Then, $V_2$ is defined exactly as $V_1$, with $U_1, Y_1$ replaced by $U_2, Y_2$. We then proceed similarly for other values of $k$ to obtain a sequence $(U_k, V_k)$ with $U_k \leq V_k \leq U_{k+1}$ for each $k$.

In case there exists ${k_\star}$ for which $U_{k_\star} < \infty$ (so that there is a bifurcation connecting $(Y_{k_\star}, U_{k_\star}-3)$ to the two points $(Y_{k_\star}-Le_1,U_{k_\star})$ and $(Y_{k_\star} + Le_1, U_{k_\star})$) and $V_{k_\star} = \infty$ (so that $(Y_{k_\star}-Le_1,U_{k_\star})\rsa \infty$ or $(Y_{k_\star} + Le_1, U_{k_\star})\rsa \infty$, or both), we let $U_\star = U_{k_\star}$ and $Y_\star = Y_{k_\star}$. Otherwise, we let $U_\star = \infty$ and $Y_\star = \triangle$.

We now state two lemmas that will be needed about these random times; the proofs are postponed to Section \ref{s:proofs_5} in the Appendix.

\begin{lemma}
\label{lem:properties_bifurcation}
Conditioned on $\{(0,0)\rightsquigarrow \infty\}$, $U_\star$ is almost surely finite:
\begin{equation}\label{eq:U_star_finite}
\P[U_\star < \infty \mid (0,0) \rsa \infty] = \sum_{k=0}^\infty \P[U_k < \infty,\; V_k = \infty \mid (0,0) \rsa \infty] = 1.
\end{equation}
Moreover, there exists $\sigma_1 > 0$ (depending on $L$) such that
\begin{equation}\label{eq:tail_U_star}
{\mathbb{E}}\left[\exp(\sigma_1\cdot  U_\star) \mid (0,0) \rsa \infty\right] < \infty.
\end{equation} 
\end{lemma}

\begin{lemma}
\label{lem:fit_U_star}
Given events $E_1,E_2$ on Harris systems,
\begin{equation}\label{eq:distr_ingr_1}\begin{split}
&{\mathbb{P}}\left[\mathbb{H}_{[0,U_\star]} \in E_1,\;  \mathbb{H} \circ \theta(Y_\star,U_\star) \in E_2 \mid (0,0) \rsa \infty\right]\\& ={\mathbb{P}}\left[\mathbb{H}_{[0,U_\star]} \in E_1 \mid (0,0) \rsa \infty\right]\cdot {\mathbb{P}}\left[ \mathbb{H}\in E_2 \mid \{(-L,0)\rightsquigarrow \infty\} \cup \{(L,0)\rightsquigarrow \infty\} \right].
\end{split}
\end{equation}
\end{lemma}

\subsubsection{Ingredient 2: Survival time for one ancestry out of a pair}
The second ingredient in our construction is another random space-time point $(Z_\star, W_\star)$ obtained as a function of the augmented Harris system $\mathbb{H}$ (again, it will depend on $H$ and not $\mathcal{H}$). The definition of this space-time point will refer to the same constant $L$ that was used in the previous subsection. For now we only assume $L \in \N$ and $L > 1$, but Lemma \ref{lem:choice_of_L} will require $L$ to be chosen large.

We assume a Harris system $H$ is given and define 
$$W_0 \equiv 0,\qquad Z_0 \equiv 0,\qquad W_1 = \sup\{t:(Le_1,0) \rsa \Z^d \times \{t\}\}.$$
 We then let 
$$Z_1 = \begin{cases} \eta^{(-Le_1,0)}_{W_1}&\text{if } W_1 < \infty;\\[.2cm]\triangle&\text{otherwise,}\end{cases}$$
and inductively, for $k \geq 1$,
\begin{align*} &W_{k+1}  = \begin{cases} \sup\{t: (Z_k,W_k) \rsa \Z^d \times \{t\}\}&\text{if } W_k < \infty,\;Z_k \neq \triangle;\\[.2cm] \infty&\text{otherwise,}\end{cases} \\[.2cm] &Z_{k+1} = \begin{cases}  \eta^{(-L e_1,0)}_{W_{k+1}}&\text{if } W_{k+1} < \infty,\\[.2cm]\triangle&\text{otherwise.}\end{cases}\end{align*}
Now, if $(Le_1,0) \rsa \infty$, so that $W_1 = \infty$, we define $(Z_\star, W_\star) = (Le_1,0)$. If, on the other hand, there exists a (necessarily unique) $k_\star \geq 1$ such that $W_{k_\star}< \infty$, $Z_{k_\star} \neq \triangle$ and $W_{{k_\star}+1} = \infty$ (so that $(Z_{k_\star},W_{k_\star}) \rsa \infty$), we let $(Z_\star, W_\star) = (Z_{k_\star},W_{k_\star})$. In all other cases, we simply put $(Z_\star, W_\star) = (\triangle, \infty)$. These definitions are illustrated in Figure \ref{fig:pair}.

\begin{figure}[htb]
\begin{center}
\setlength\fboxsep{0pt}
\setlength\fboxrule{0.5pt}
\fbox{\includegraphics[width = 0.5\textwidth]{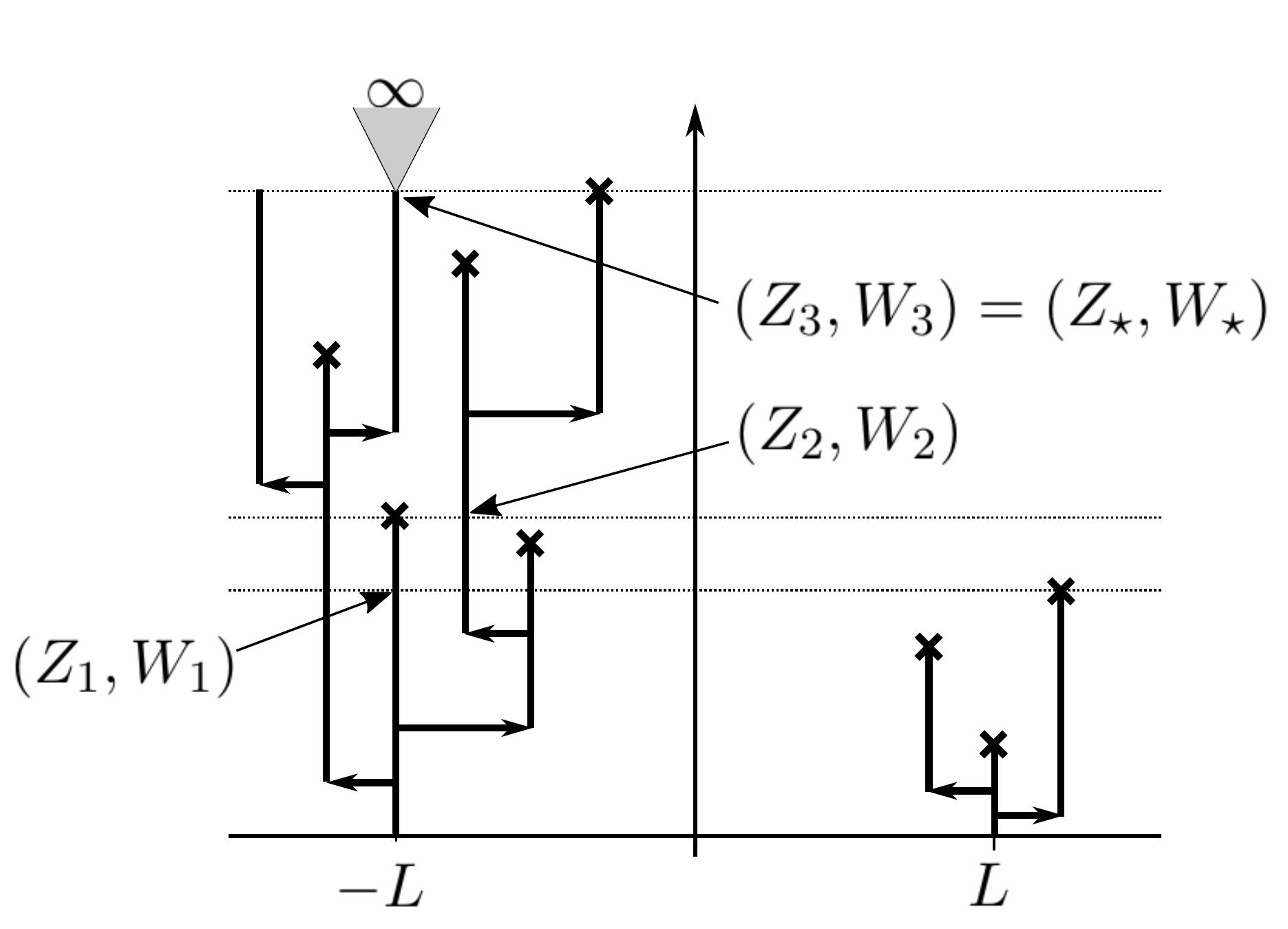}}
\end{center}
\caption{The space-time points $(Z_k, W_k)$ and $(Z_\star, W_\star)$ (dimension one). The grey triangle with the infinity symbol indicates that $(Z_3, W_3) \rsa \infty$.}
\label{fig:pair}
\end{figure}

\begin{lemma} Conditioned on $\{(Le_1,0) \rsa \infty\} \cup \{(-Le_1,0)\rsa \infty\}$, $W_\star$ is almost surely finite:
\begin{equation}\begin{split}
&\P\left[W_\star < \infty \mid \{(Le_1,0) \rsa \infty\} \cup \{(-Le_1,0)\rsa \infty\} \right] \\[.2cm]&\hspace{2cm}= \P\left[W_\star = 0 \mid \{(Le_1,0) \rsa \infty\} \cup \{(-Le_1,0)\rsa \infty\}\right]\\& \hspace{2.5cm}+ \sum_{k=1}^\infty \P\left[W_\star = W_k < \infty \mid \{(Le_1,0) \rsa \infty\} \cup \{(-Le_1,0)\rsa \infty\}\right] = 1;\end{split}
\end{equation}
moreover, there exists a constant $\sigma_2 > 0$, independent of $L$, such that
\begin{equation}\label{eq:tail_hat}
{\mathbb{E}}\left[\exp(\sigma_2 \cdot W_\star) \mid \{(Le_1,0) \rsa \infty\} \cup \{(-Le_1,0)\rsa \infty\}\right] < \infty.
\end{equation}
\end{lemma}
The proof of this result is similar to that of Lemma \ref{eq:U_star_finite}, only simpler, so we will omit it.

\begin{lemma}\label{lem:law_ing2}
Given events $E_1$ and $E_2$ on Harris systems,
\begin{equation*}\begin{split}
&\P\left[\mathbb{H}_{[0,W_\star]} \in E_1,\; \mathbb{H}\circ \theta(Z_\star, W_\star) \in E_2 \mid  \{(Le_1,0) \rsa \infty\} \cup \{(-Le_1,0)\rsa \infty\}\right] \\[.2cm] &= \P\left[\mathbb{H}_{[0,W_\star]} \in E_1 \mid  \{(Le_1,0) \rsa \infty\} \cup \{(-Le_1,0)\rsa \infty\}\right]\\[.2cm] &\hspace{4.5cm} \cdot \P\left[\mathbb{H} \in E_2 \mid  \{(Le_1,0) \rsa \infty\} \cup \{(-Le_1,0)\rsa \infty\}\right].\end{split}
\end{equation*}
\end{lemma}
Lemma \ref{lem:law_ing2} is proved in Section \ref{s:proofs_5} in the Appendix.

\begin{lemma}
\label{lem:choice_of_L} If $L$ is large enough, $${\E}[Z_\star \mid \{(Le_1,0) \rsa \infty\} \cup \{(-Le_1,0) \rsa \infty\}] = \alpha_L \cdot e_1,$$ where $\alpha_L > 0$.
\end{lemma}
\begin{proof}
We abbreviate
$$A_- = \{(-Le_1,0) \rsa \infty\},\qquad A_+ = \{(Le_1,0) \rsa \infty\},\qquad A = A_- \cup A_+.$$
We first need to prove that
\begin{equation} \label{eq:Z_star_integrable}
\E\left[ \|Z_\star\| \cdot \mathds{1}_{A}\right] < \infty.
\end{equation} 
To this end, we write $A = A_+ \cup (A_+^c \cap A_-)$, so 
\begin{equation}\label{eq:Z_star_decomp}\begin{split} &Z_\star \cdot \mathds{1}_{A} = Le_1 \cdot \mathds{1}_{A_+} - Le_1 \cdot \mathds{1}_{A_+^c \cap A_-} + (Z_\star + Le_1)  \cdot \mathds{1}_{A^c_+ \cap A_-}.   \end{split}\end{equation}
Because of  this equality, \eqref{eq:Z_star_integrable} will follow from the statement:
\begin{equation} \label{eq:Z_star_need2}\exists C > 0: \; \forall L > 0,\;\;\E\left[ \|Z_\star + Le_1\| \cdot \mathds{1}_{A_+^c \cap A_-}\right] < C.\end{equation}
Let us prove \eqref{eq:Z_star_need2}. The expectation is equal to
$$\E\left[ \left\|\eta^{(-Le_1,0)}_{W_\star} + Le_1\right\| \cdot \mathds{1}_{A_+^c \cap A_-}\right].$$
Define
$$M_t = \sup\{\|z+ Le_1\|:\; z \in \Z^d,\; \exists s \leq t:\; (-Le_1,0) \rsa (z,s)\},\quad t \geq 0.$$
For any $x > 0$  and $\alpha \in (0,1)$ we have
\begin{equation}\label{eq:aux_tail00}\begin{split}
\P\left[ \left\|\eta^{(-Le_1,0)}_{W_\star} + Le_1\right\| \cdot \mathds{1}_{A_+^c \cap A_-} > x\right] \leq &\P\left[ \{W_\star > \alpha x\} \cap A_+^c \cap A_- \right]+ \P\left[M_{\alpha x} > x\right].
\end{split}\end{equation}
Using \eqref{eq:tail_hat} and the Chebyshev's inequality,
\begin{equation}\label{eq:aux_tail10}\begin{split}\P\left[ \{W_\star > \alpha x\} \cap A_+^c \cap A_- \right] &\leq \P\left[ \{W_\star > \alpha x\} \cap A_+^c \cap A_- \mid A\right] \\&\leq  \P\left[ W_\star > \alpha x\mid A\right]< C\exp(-\sigma_2 \alpha x), \end{split}\end{equation}
where $C$ is a constant that does not depend on $L$. Next, by \eqref{eq:1cp_moves_slowly},
\begin{equation}\label{eq:aux_tail11}
\P\left[M_{\alpha x} > x\right] \leq \exp(b_1 \alpha x - b_2 x).
\end{equation}
Using \eqref{eq:aux_tail10} and  \eqref{eq:aux_tail11} in \eqref{eq:aux_tail00} with $\alpha < b_2/b_1$ and integrating over $x$, we obtain \eqref{eq:Z_star_need2}, so also \eqref{eq:Z_star_integrable}.

By symmetry, we have $\mathbb{E}[Z_\star \cdot e_i] = 0$ for all $i \neq 1$. To treat the expectation of $Z_\star \cdot e_1$, we use \eqref{eq:Z_star_decomp} to decompose:
\begin{align*}
&\E\left[(Z_\star \cdot e_1) \mathds{1}_{A}\right] \\
&=  L \cdot \P(A_+) - L \cdot \P(A_+^c \cap A_-) + \E\left[ (Z_\star\cdot e_1 + L) \cdot \mathds{1}_{A_+^c \cap A_-}\right]\\
&=  L \cdot \P(A_+) - L \cdot \P(A_+ \cap A_-^c) + \E\left[ (Z_\star\cdot e_1 + L) \cdot \mathds{1}_{A_+^c \cap A_-}\right]\\
&= L \cdot \P[A_+ \cap A_-] +  \E\left[ (Z_\star \cdot e_1 + L) \cdot \mathds{1}_{A_+^c \cap A_-}\right]\\
&\geq L \cdot \P[(0,0)\rsa \infty]^2 +  \E\left[ (Z_\star\cdot e_1 + L) \cdot \mathds{1}_{A_+^c \cap A_-}\right]
\end{align*}
where the last step follows from the FKG inequality and translation invariance. By \eqref{eq:Z_star_need2}, we can now choose $L$ large enough that the right-hand side is positive, completing the proof. 
\end{proof}

We end this subsection defining $(Z_k', W_k')_{k\geq 0},\; (Z_\star',W_\star')$ exactly as $(Z_k, W_k)_{k \geq 1},\;(Z_\star, W_\star)$, but inverting the roles of $Le_1$ and $-Le_1$. More formally, we define
$$(Z_k',W_k') = (Z_k,W_k) \circ \psi^{1,-1},\qquad (Z_\star',W_\star') = (Z_\star,W_\star) \circ \psi^{1,-1}$$
where $\psi^{1,-1}$ is as defined in \eqref{eq:linear_psi}, that is, it is the linear transformation $\psi^{1,-1}(x_1,x_2,\ldots, x_d) = (-x_1,x_2,\ldots, x_d)$. Choosing $L$ large enough as required by Lemma \ref{lem:choice_of_L}, we then have
\begin{equation}
{\E}[Z_\star' \mid \{(-Le_1,0) \rsa \infty \} \cup \{(Le_1,0) \rsa \infty\}] = -\alpha_L e_1.
\end{equation}

\subsubsection{Putting ingredients together}
In what follows, we always assume that the event $\{(0,0) \rightsquigarrow \infty\}$ occurs. By Lemma \ref{lem:properties_bifurcation}, it follows from this assumption that $U_\star < \infty$, $Y_\star \in \mathbb{Z}^d$.

Recall that, if $U_\star$ is a bifurcation time, then there is exactly one  arrow from $\{Y_\star\} \times [U_\star-3,U_\star-2]$ to  $\{Y_\star-e_1\} \times [U_\star-3,U_\star-2]$ and one  arrow from $\{Y_\star\} \times [U_\star-3,U_\star-2]$ to  $\{Y_\star+e_1\} \times [U_\star-3,U_\star-2]$. Let $t_{(-)}, t_{(+)} \in [U_\star -3,U_\star - 2]$ be the respective times at which these arrows are present. Define 
\begin{equation*}
E = \left\{\text{for some $t \in [U_\star-2, U_\star-1]$, there is a selective arrow from $(Y_\star,t)$ to $(Y_\star + e_1,t)$}\right\}.
\end{equation*}
If $E$ occurs, let $t_{(+)}'\in [U_\star-2, U_\star-1]$ be the first time in $[U_\star -2, U_\star -1]$ at which a selective arrow as mentioned in $E$ appears.

We now define the random variables $\mathcal{T}$ and $\mathcal{X}$ in the statement of Proposition \ref{prop:steer}:
\begin{align} \label{eq:def_of_final_rvs}
(\mathcal{X},\mathcal{T}) = \begin{cases} (Y_\star, U_\star)  + (Z_\star, W_\star) \circ \theta(Y_\star, U_\star) & \text{on } E \cup \{t_{(+)} > t_{(-)}\};\\[.2cm](Y_\star, U_\star)  + (Z'_\star, W'_\star) \circ \theta(Y_\star, U_\star)&\text{on } E^c \cap \{t_{(+)} < t_{(-)}\}.\end{cases}\end{align}
 
\begin{proof}[Proof of Proposition \ref{prop:steer}(1)] 
We will construct an RFSIP $\upgamma_\star: [0,\mathcal{T}] \to \Z^d$ with $\upgamma_\star(0) = 0$ and $\upgamma_\star(\mathcal{T})= \mathcal{X}$. The definition of $\upgamma_\star$ will be split into six cases. In each case, it will be straightforward to verify that $\upgamma_\star$ is either an RFBIP (cases 1-4 and 6) or an RFSIP (case 5). Figure \ref{fig:cases} serves as a useful guide to the construction. 

In all cases, on $[0,U_\star-3]$ we let $\upgamma_\star$ be equal to the unique RFBIP from $(0,0)$ to $(Y_\star, U_\star-3)$ (such an RFBIP exists by the definition of the ancestor process, since $\eta_{U_\star - 3} =  Y_\star$).

\begin{figure}[htb]
\begin{center}
\setlength\fboxsep{0pt}
\setlength\fboxrule{0.5pt}
\fbox{\includegraphics[width = 0.7\textwidth]{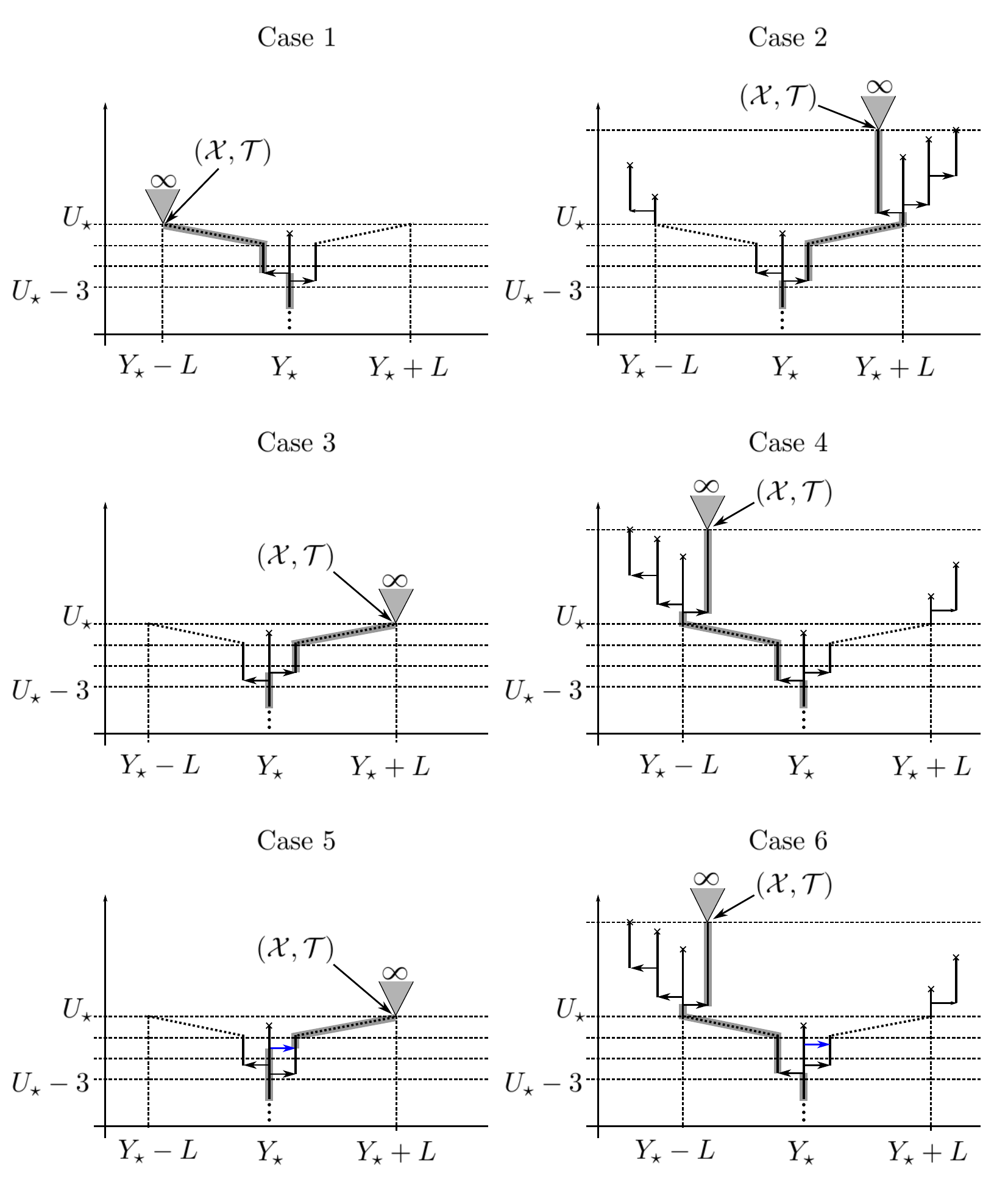}}
\end{center}
\caption{The six cases in the definition of the RFSIP $\upgamma_\star$ from $(0,0)$ to $(\mathcal{X},\mathcal{T})$ (dimension one). The path $\upgamma_\star$ is shown as a thick grey line. A grey triangle with the infinity sign at a space-time point $(x,t)$ indicates that $(x,t)\rightsquigarrow \infty$. In each case one verifies that $\upgamma_\star$ is an RFSIP because whenever it jumps, that is, whenever $\upgamma_\star(t-) = x \neq y = \upgamma_\star(t)$, we have $(x,t+) \not \rightsquigarrow \mathcal{T}$.}
\label{fig:cases}
\end{figure}

\begin{itemize}
\item[Case 1:] $E^c \cap \{t_{(-)} > t_{(+)}\} \cap \{(Y_\star -Le_1,U_\star) \rightsquigarrow \infty\}$ occurs. In this case, we have
\begin{equation}\label{eq:one_of_the_cases}
(\mathcal{X}, \mathcal{T})= (Y_\star, U_\star) + (Z'_\star, W'_\star)\circ \theta(Y_\star, U_\star) 
\end{equation}
and in the translated Harris system $H \circ \theta (Y_\star, U_\star)$ we have $(-Le_1,0) \rightsquigarrow \infty$ and hence $W'_\star \circ \theta(Y_\star, U_\star) = 0$ and $Z'_\star \circ \theta (Y_\star, U_\star) = -Le_1$. Hence, $(\mathcal{X}, \mathcal{T}) =  (Y_\star - Le_1, U_\star)$. Then, we set:
\begin{itemize}
\item[$\bullet$] on $[U_\star - 3,\; \tmin)$, $\upgamma_\star$ equal to $Y_\star$;
\item[$\bullet$] on $[\tmin,\; U_\star] = [\tmin,\;\mathcal{T}]$, $\upgamma_\star$ equal to the unique RFBIP from $(Y_\star -e_1,\; \tmin)$ to $(Y_\star - Le_1,\;U_\star)$.
\end{itemize}
\item[Case 2:] $E^c \cap \{\tmin > \tplus\} \cap \{(Y_\star -Le_1,U_\star) \not\rightsquigarrow \infty\}$ occurs. We again have \eqref{eq:one_of_the_cases}, but now $W'_\star \circ \theta(Y_\star, U_\star) > 0$, so that $\mathcal{T} > U_\star$, and $\mathcal{X} = \eta^{(Y_\star + Le_1, U_\star)}_{\mathcal{T}}$, by the definition of $Z'_\star$. In particular, there is a unique RFBIP from $(Y_\star + Le_1, U_\star)$ to $(\mathcal{X},\mathcal{T})$. 
We set
\begin{itemize}
\item[$\bullet$] on $[U_\star - 3,\; \tplus)$, $\upgamma_\star$ equal to $Y_\star$;
\item[$\bullet$] on $[\tplus,\; U_\star)$, $\upgamma_\star$ equal to the unique RFBIP from $(Y_\star+e_1,\; \tplus)$ to $(Y_\star + Le_1,\;U_\star)$;
\item[$\bullet$] on $[U_\star, \mathcal{T}]$, $\upgamma_\star$ equal to the unique RFBIP from $(Y_\star + Le_1, U_\star)$ to $(\mathcal{X}, \mathcal{T})$.
\end{itemize}
\item[Case 3:] $E^c \cap \{\tplus > \tmin\} \cap \{(Y_\star + Le_1,U_\star) \rightsquigarrow \infty\}$ occurs. We have
\begin{equation} \label{eq:case_S_r}
(\mathcal{X}, \mathcal{T})= (Y_\star, U_\star) + (Z_\star, W_\star)\circ \theta(Y_\star, U_\star)
\end{equation}
and in the translated Harris system $H \circ \theta(Y_\star,U_\star)$ we have $(Le_1,0) \rsa \infty$, so $(\mathcal{X},\mathcal{T}) = (Y_\star + Le_1, U_\star)$. We set:
\begin{itemize}
\item[$\bullet$] on $[U_\star - 3,\; \tplus)$, $\upgamma_\star$ equal to $Y_\star$;
\item[$\bullet$] on $[\tplus,\; U_\star] = [\tplus,\;\mathcal{T}]$, $\upgamma_\star$ equal to the unique RFBIP from $(Y_\star +e_1,\; \tplus)$ to $(Y_\star + Le_1,\;U_\star)$.
\end{itemize}
\item[Case 4:] $E^c \cap \{\tplus > \tmin\} \cap \{(Y_\star + Le_1,U_\star) \not \rightsquigarrow \infty\}$ occurs. We again have \eqref{eq:case_S_r}, but now
$\mathcal{T}$ is some time larger than $U_\star$ and $\mathcal{X} = \eta^{(Y_\star - Le_1,U_\star)}_{\mathcal{T}}$. We set:
\begin{itemize}
\item[$\bullet$] on $[U_\star - 3,\; \tmin)$, $\upgamma_\star$ equal to $Y_\star$;
\item[$\bullet$] on $[\tmin,\; U_\star)$, $\upgamma_\star$ equal to the unique RFBIP from $(Y_\star-e_1,\; \tmin)$ to $(Y_\star - Le_1,\;U_\star)$;
\item[$\bullet$] on $[U_\star, \mathcal{T}]$, $\upgamma_\star$ equal to the unique RFBIP from $(Y_\star - Le_1, U_\star)$ to $(\mathcal{X}, \mathcal{T})$.
\end{itemize}
\item[Case 5:] $E \cap \{(Y_\star + Le_1,U_\star)  \rightsquigarrow \infty\}$. In this case we again have \eqref{eq:case_S_r}. We define $\upgamma_\star$ as in Case 3, with the only difference that we replace $\tplus$ by $\tplus'$ everywhere.
\item[Case 6:] $E \cap \{(Y_\star + Le_1,U_\star)  \not\rightsquigarrow \infty\}$. As in Case 4, we have  \eqref{eq:case_S_r} with $\mathcal{T} > U_\star$ and $\mathcal{X} = \eta^{(Y_\star -Le_1,U_\star)}_{\mathcal{T}}$. We define $\upgamma_\star$ exactly as in Case 4.
\end{itemize}

\end{proof}

\begin{proof}[Proof of Proposition \ref{prop:steer}(2)-(4)]
Statement (2) follows directly from Lemmas \ref{lem:fit_U_star} and \ref{lem:law_ing2}. The fact that 
$$\mathbb{E}\left[\exp(\sigma \mathcal{T}) \mid (0,0) \rsa \infty\right] < \infty$$
if $\sigma$ is small enough follows from  \eqref{eq:tail_U_star}, Lemma \ref{lem:fit_U_star} and \eqref{eq:tail_hat}. Now, let
$$M_t = \sup\{\|x\|: x \in \Z^d,\; \text{there is an SIP from $(0,0)$ to $(x,s)$ for some } s \leq t \},\quad t \geq 0.$$
Noting that there is an SIP from $(0,0)$ to $(\mathcal{X}, \mathcal{T})$ we have, for any $\alpha \in (0,1)$ and $x > 0$,
\begin{align*}
\mathbb{P}\left[\mathcal{X} \cdot \mathds{1}_{\{(0,0) \rsa \infty\}} > x\right] \leq & \P\left[(0,0) \rsa \infty,\; \mathcal{T} > \alpha x\right] + \P\left[M_{\alpha x} > x \right].
\end{align*}
We can then bound the two terms on the right-hand side as we did in \eqref{eq:aux_tail10} and \eqref{eq:aux_tail11} to show that, if $\sigma$ is small enough,
$$\mathbb{E}\left[\exp(\sigma \mathcal{X}) \mid (0,0) \rsa \infty\right] < \infty.$$
Hence, (3) is proved. 

We now turn to (4). We abbreviate
\begin{align*}&\widetilde{\mathbb{P}}(\cdot) = \mathbb{P}(\cdot \mid (0,0) \rsa \infty),\quad \widetilde{\mathbb{E}}(\cdot) = \mathbb{E}(\cdot \mid (0,0) \rsa \infty),\\ &\widehat{\mathbb{E}}(\cdot) = \mathbb{E}(\cdot \mid \{(-Le_1,0) \rsa \infty\} \cup \{(Le_1,0) \rsa \infty\}). \end{align*}
We start with the equalities
\begin{equation} \label{eq:ingredients1} \begin{split}
\widetilde{\E}[\mathds{1}_{E \cup \{t_{(+)} > t_{(-)}\}} \cdot \mathcal{X}] &\stackrel{\eqref{eq:def_of_final_rvs}}{=} \widetilde{\E}[\mathds{1}_{E \cup \{t_{(+)} > t_{(-)}\}} \cdot (Y_\star+ Z_\star \circ \theta(Y_\star, U_\star))]\\[.2cm]
&\stackrel{\eqref{eq:distr_ingr_1}}{=} \widetilde{\E}[\mathds{1}_{E \cup \{t_{(+)} > t_{(-)}\}} \cdot Y_\star] + \widetilde{\P}[E \cup \{t_{(+)} > t_{(-)}\}] \cdot \widehat{\E}[Z_\star];\end{split}
\end{equation}
\begin{equation} \label{eq:ingredients2} \begin{split}
\widetilde{\E}[\mathds{1}_{E^c \cap \{t_{(+)} < t_{(-)}\}} \cdot \mathcal{X}] &\stackrel{\eqref{eq:def_of_final_rvs}}{=} \widetilde{\E}[\mathds{1}_{E^c \cap \{t_{(+)} < t_{(-)}\}} \cdot (Y_\star+ Z'_\star \circ \theta(Y_\star, U_\star))]\\[.2cm]
&\stackrel{\eqref{eq:distr_ingr_1}}{=} \widetilde{\E}[\mathds{1}_{E^c \cap \{t_{(+)} < t_{(-)}\}} \cdot Y_\star] + \widetilde{\P}[E^c \cap \{t_{(+)} < t_{(-)}\}] \cdot \widehat{\E}[Z'_\star];\end{split}
\end{equation}
By symmetry, $\widetilde{\E}[Y_\star] = 0$ and $\widehat{\E}[Z_\star] = - \widehat{\E}[Z'_\star]$, so adding together \eqref{eq:ingredients1} and \eqref{eq:ingredients2} yields
\begin{equation} \label{eq:ingredients3}\widetilde{\E}[\mathcal{X}] = \widehat{\E}[Z_\star] \cdot (\widetilde{\P}[E \cup \{t_{(+)} > t_{(-)}\}] - \widetilde{\P}[E^c \cap \{t_{(+)} < t_{(-)}\}]).\end{equation}
Noting that $E$ only depends on the presence of a selective arrow on $[U_\star-2, U_\star -1]$ and again using symmetry, we have
\begin{align*} &\widetilde{\P}[E^c \cap \{t_{(+)} < t_{(-)}\}] = e^{-(\lambda_1 - \lambda_2)} \cdot \widetilde{\P}[t_{(+)} < t_{(-)}]= \frac12 \cdot e^{-(\lambda_1 - \lambda_2)} < \frac12 \\[.2cm]&\hspace{7cm} \Longrightarrow \quad \widetilde{\P}[E \cup \{t_{(+)} > t_{(-)}\}] > \frac12.
\end{align*}
Using this and the fact that $\widehat{\E}[Z_\star] > 0$ in \eqref{eq:ingredients3} concludes the proof.
\end{proof}

\section{Appendix}\label{s:appendix}
\subsection{Proofs of results of Section \ref{s:times}}
\label{s:proofs_5}
We let $(\scF_t)_{t \geq 0}$ denote the natural filtration of the Poisson point processes in $\mathbb{H}$ (that is, for each $t$, $\scF_t$ is the $\sigma$-algebra generated by $\mathbb{H}_{[0,t]}$). Before turning to the statements of Section \ref{s:times}, we state and prove a preliminary result.
\begin{lemma} \label{lem:behavior_EU}
There exists $\sigma_0 > 0$ such that
\begin{equation}\label{eq:behavior_EU_init}
\mathbb{E}\left[\exp({\sigma_0 U_1})\cdot \mathds{1}_{\{U_1 < \infty\}} \right] < \infty
\end{equation}
and, for any $k$, on $\{V_k < \infty,\; \eta_{V_k} \neq \triangle\}$,
\begin{equation}\label{eq:behavior_EU}
\mathbb{E}\left[\exp({\sigma_0 (U_{k+1} - V_k)})\cdot \mathds{1}_{\{U_{k+1} < \infty\}} \mid \scF_{V_k}\right] < \infty.
\end{equation}
\end{lemma}
\begin{proof}
Let $\mathcal{E}$ be the set of augmented Harris systems for which $t = 3$ is a bifurcation time. We have $\P(\mathcal{E}) > 0$, since the occurrence of $\mathcal{E}$ can be guaranteed by making prescriptions on finitely many Poisson processes on the time interval $[0,3]$.

For $t \geq 3$, we have
\begin{align*}
\P \left[ \eta_t \neq \triangle,\; U_1 > t\right] &\leq \P\left[ \eta_{t-3}\neq \triangle,\;U_1 > t-3,\; \mathbb{H}\circ \theta(\eta_{t-3},t-3) \notin \mathcal{E} \right]\\
&= \P\left[ \eta_{t-3}\neq \triangle,\;U_1 > t-3 \right] \cdot \P(\mathcal{E}^c),
\end{align*}
and iterating we show that the right-hand side is less than $\P(\mathcal{E}^c)^{\lfloor t/3\rfloor}$, proving that, if $c > 0$ is small enough,
\begin{equation}\label{eq:behavior_U_init}
\P \left[ \eta_t \neq \triangle, \; U_1 > t\right] < e^{-ct},\qquad t \geq 3.
\end{equation}
Then, noting that $\{U_1 < \infty\} \subset \{\eta_t \neq \triangle\;\forall t < U_1\}$,
\begin{align*}
\E\left[e^{\sigma U_1}\cdot \mathds{1}_{\{U_1 < \infty \}} \right] &= \int_0^\infty \P \left[e^{\sigma U_1}\cdot \mathds{1}_{\{U_1 < \infty \}} > x \right] \mathsf{d} x \\
&\leq \int_0^\infty \P\left[ U_1 > \frac{\log(x)}{\sigma},\; \eta_t \neq \triangle\; \forall t < U_1\right] \mathsf{d} x
\stackrel{\eqref{eq:behavior_U_init}}{ \leq} \int_0^\infty e^{-\frac{c}{\sigma} \log(x)}\mathsf{d} x < \infty
\end{align*}
if $\sigma \in (0,c)$.

To prove \eqref{eq:behavior_EU}, we argue as above (also  using the strong Markov property with respect to the stopping time $V_k$) to obtain that, for any $k$, on the event $\{ V_k < \infty,\; \eta_{V_k }\neq \triangle\}$, 
\begin{equation}\label{eq:behavior_U}
\P\left[ \eta_{V_k + t} \neq \triangle,\; U_{k+1} > V_k + t \mid \scF_{V_k}\right] < e^{-ct},\qquad t > 0;
\end{equation}
we then complete the proof as above.
\end{proof}

\begin{proof}[Proof of Lemma \ref{lem:properties_bifurcation}] To prove \eqref{eq:U_star_finite}, start noting that, for all $k \geq 1$,
\begin{align*}
&\P\left[ U_{k+1} = \infty \mid V_k < \infty,\; (0,0) \rsa \infty\right] \\\hspace{2cm}&= \P[(0,0) \rightsquigarrow \infty,\; V_k < \infty]^{-1} \lim_{t \to \infty} \P[(0,0) \rightsquigarrow \infty,\;V_k < \infty,\;U_{k+1} > V_k + t]\\[.2cm]
&\leq \P[(0,0) \rightsquigarrow \infty,\; V_k < \infty]^{-1} \lim_{t \to \infty} \P[ V_k < \infty,\;\eta_{V_k + t} \neq \triangle,\;U_{k+1} > V_k + t] \stackrel{\eqref{eq:behavior_U}}{=} 0.
\end{align*}
Similarly, by \eqref{eq:behavior_U_init},
$$\P[U_1 = \infty \mid (0,0) \rsa \infty] = 0.$$
Next, for $k \geq 1$,
\begin{align*}\P[V_k < \infty] &= \mathbb{E}\left[\mathds{1}_{\{U_k < \infty\}} \cdot \P[V_k < \infty\mid \scF_{U_k}]\right] \\&= \P[U_k < \infty] \cdot \P\left[ (-Le_1,0) \not \rsa \infty,\; (Le_1,0) \not \rsa \infty\right],
\end{align*}
and iterating,
$$\P[V_k < \infty] \leq  \P\left[ (-Le_1,0) \not \rsa \infty,\; (Le_1,0) \not \rsa \infty\right]^k,$$
so
\begin{align*}&\P[V_k < \infty \; \forall k ]  = 0 \quad \Longrightarrow \quad\P[V_k < \infty \; \forall k \mid (0,0) \rsa \infty]  = 0.
\end{align*}
Putting these facts together, we see that conditionally to $\{(0,0) \rsa \infty\}$, almost surely there exists $k$ such that $U_k < \infty$ and $V_k = \infty$, completing the proof of \eqref{eq:U_star_finite}.

We now turn to \eqref{eq:tail_U_star}. 
Fix $\sigma > 0$. By \eqref{eq:U_star_finite} we have
\begin{equation}\label{eq:sum_U_star}
 \E[e^{\sigma U_\star} \mid (0,0) \rsa \infty] \leq \sum_{k=1}^\infty \E\left[e^{\sigma U_k}\cdot \mathds{1}_{\{U_k < \infty\}} \right].
\end{equation}
Fix $k \geq 2$. We have
\begin{align*}
\E\left[e^{\sigma U_k}\cdot \mathds{1}_{\{U_k < \infty\}} \right] &= \E\left[e^{\sigma V_{k-1}} \cdot \mathds{1}_{\{V_{k-1}< \infty, \;\eta_{V_{k-1}}\neq \triangle\}}\cdot \E\left[ e^{\sigma(U_k - V_{k-1})}\cdot \mathds{1}_{\{U_k < \infty \}} \mid \scF_{V_{k-1}}\right]\right]\\
&\leq C_\sigma \cdot \E \left[e^{\sigma V_{k-1}} \cdot \mathds{1}_{\{V_{k-1} < \infty\}} \right]
\end{align*}
for some $C_\sigma \in (0,\infty)$ if $\sigma$ is small enough, by \eqref{eq:behavior_EU}. Next, we have
\begin{align*}
\E \left[e^{\sigma V_{k-1}} \cdot \mathds{1}_{\{V_{k-1} < \infty\}} \right] &= \E\left[e^{\sigma U_{k-1}}\cdot \mathds{1}_{\{U_{k-1} < \infty \}} \cdot \E\left[e^{\sigma (V_{k-1} - U_{k-1})}\cdot \mathds{1}_{\{V_{k-1}<\infty\}} \mid \scF_{U_{k-1}}\right]\right]\\
&= g(\sigma) \cdot \E\left[e^{\sigma U_{k-1}}\cdot \mathds{1}_{\{U_{k-1} < \infty \}} \right],
\end{align*}
where
$$g(\sigma) = \mathbb{E}\left[ \exp\left(\sigma \cdot T^{\{-Le_1,Le_1\}} \right)\cdot \mathds{1}_{\{T^{\{-Le_1,Le_1\} } < \infty \} }\right],$$
with $T^{\{-Le_1,Le_1\}}$ as in \eqref{eq:definition_death_time}. Iterating these bounds, we obtain
\begin{equation}\E\left[e^{\sigma U_k}\cdot \mathds{1}_{\{U_k < \infty\}} \right] < (C_\sigma \cdot g(\sigma))^k,\qquad k \geq 2.\label{eq:bound_geometric}\end{equation}
Now, by the Dominated Convergence Theorem, \eqref{eq:dies_quickly}, and the fact that 
$$\P\left[ T^{\{-Le_1,Le_1\}} < \infty\right] < 1,$$
we can reduce $\sigma$ so that $g(\sigma) < 1$, and then reduce it further so that $C_\sigma \cdot g(\sigma) < 1$. By \eqref{eq:behavior_EU_init}, \eqref{eq:sum_U_star}, and \eqref{eq:bound_geometric}, the proof of \eqref{eq:tail_U_star} is complete.
\end{proof}

\begin{proof}[Proof of Lemma \ref{lem:fit_U_star}] 
\begin{align*}
&\P\left[ \mathbb{H}_{[0,U_\star]} \in E_1,\;  \mathbb{H} \circ \theta(Y_\star,U_\star) \in E_2 \mid (0,0) \rsa \infty\right] \\&= \sum_{y \in \Z^d}\sum_{k=1}^\infty \P[(0,0)\rightsquigarrow \infty]^{-1} \cdot \P\left[\begin{array}{l}U_k < \infty,\; Y_k = y,\; \mathbb{H}_{[0,U_k]} \in E_1,\\[.3cm] \; (y-Le_1,U_k) \rsa \infty \text{ or } (y+Le_1,U_k) \rsa \infty,\\[.3cm] \mathbb{H} \circ \theta(y, U_k) \in E_2\end{array} \right]\\[.2cm]
&= \P\left[\mathbb{H} \in E_2 \mid \{(-Le_1,0) \rightsquigarrow \infty\} \cup\{(Le_1,0) \rightsquigarrow \infty\} \right]\\
&\hspace{2cm} \cdot \sum_{y \in \Z^d}\sum_{k=1}^\infty \P[(0,0)\rightsquigarrow \infty]^{-1} \cdot \P\left[\begin{array}{l}U_k < \infty,\; Y_k = y,\; \mathbb{H}_{[0,U_k]} \in E_1,\\[.3cm] \;(y-Le_1,U_k) \rsa \infty \text{ or } (y+Le_1,U_k) \rsa \infty\end{array} \right]\\[.2cm]
&=\P\left[\mathbb{H} \in E_2 \mid \{(-Le_1,0) \rightsquigarrow \infty\} \cup\{(Le_1,0) \rightsquigarrow \infty\} \right] \cdot \P[\mathbb{H}_{[0,U_\star]} \in E_1 \mid (0,0) \rsa \infty].
\end{align*}
\end{proof}

\begin{proof}[Proof of Lemma \ref{lem:law_ing2}]  We abbreviate
\begin{align*}
&A = \{(Le_1,0)\rsa \infty\} \cup \{(-Le_1,0) \rsa \infty\};\\
&B_0 = \{(Le_1,0) \rsa \infty\},\qquad B_{k,z} = \{W_\star = W_k < \infty,\; Z_k = z\},\; k \geq 1,\;z\in \Z^d;\\
&D = \{W_\star < \infty,\; \mathbb{H}_{[0,W_\star]} \in E_1,\; \mathbb{H} \circ \theta(Z_\star, W_\star) \in E_2\}.
\end{align*}
We then have
\begin{align}
\P(D\mid A) &= \P(A)^{-1}\cdot \left(\P( D \cap B_0) + \sum_{k\geq 1, z \in \Z^d}\P( D \cap B_{k,z})\right)\label{eq:long_id}
\end{align}
Now note that
\begin{align*}
\P(D \cap B_0) &= \P\left[(Le_1,0) \rsa \infty,\; \mathbb{H}_{\{0\}} \in E_1,\; \mathbb{H} \circ \theta(Le_1,0) \in E_2\right] \\
&= \P\left[(Le_1,0) \rsa \infty,\; \mathbb{H}_{\{0\}} \in E_1\right] \cdot \P[\mathbb{H} \in E_2 \mid (0,0) \rsa \infty];
\end{align*}
here, $\mathbb{H}_{\{0\}}$ is the trivial (almost surely empty) restriction of $\mathbb{H}$ to the degenerate interval $\{0\}$. Moreover, for $k \geq 1$ and $z \in \Z^d$,
\begin{align*}
\P(D \cap B_{k,z}) &= \P\left[W_k < \infty,\; Z_k = z,\; (Z_k, W_k) \rsa \infty,\; \mathbb{H}_{[0,W_k]} \in E_1,\; \mathbb{H} \circ \theta(Z_k,W_k) \in E_2\right]\\[.2cm]
&= \P\left[ W_k < \infty,\; Z_k = z,\; (Z_k, W_k) \rsa \infty,\; \mathbb{H}_{[0,W_k]} \in E_1\right] \cdot \P[\mathbb{H} \in E_2 \mid (0,0) \rsa \infty].
\end{align*}
Using these identities, the right-hand side of \eqref{eq:long_id} is seen to be equal to
\begin{align*}
&\frac{\P[\mathbb{H} \in E_2 \mid (0,0) \rsa \infty]}{\P(A)} \cdot \left(\P\left[\begin{array}{l}(Le_1,0) \rsa \infty,\\ \mathbb{H}_{\{0\}} \in E_1 \end{array} \right] + \sum_{k\geq 1,z\in\Z^d} \P\left[\begin{array}{l}W_k < \infty, Z_k = z,\\
(Z_k,W_k) \rsa \infty,\; \mathbb{H}_{[0,W_k]} \in E_1 \end{array} \right]\right)\\[.2cm]
&= {\mathbb{P}}[\mathbb{H} \in E_2 \mid (0,0) \rsa \infty] \cdot {\P}[\mathbb{H}_{[0,W_\star]} \in E_1 \mid A],
\end{align*}
since 
$$\{(Le_1,0) \rsa \infty\} \subseteq A,\qquad \{W_k < \infty,\;Z_k = z,\; (Z_k, W_k) \rsa \infty\} \subseteq \{(-Le_1,0) \rsa \infty\} \subseteq A.$$
\end{proof}

\subsection{Proofs of results for steered random walks}
We will need the following elementary facts about sums of independent and identically distributed random variables:
\begin{lemma}\label{lem:large_deviations}
Let $Y_1, Y_2,\ldots$ be independent and identically distributed random variables, and let $Z_0 = 0$ and $Z_n = \sum_{i=1}^n Y_i$ for $n \geq 1$.
\begin{enumerate}
\item For any $\ell > 0$, letting $h^Z_{[\ell,\infty)} = \inf\{n: Z_n \geq \ell\}$,
\begin{equation}\label{eq:neat_form} \P\left[h^Z_{[\ell,\infty)}<\infty,\;Z_{h^Z_{[\ell,\infty)}} - \ell \geq x\right] \leq \left(\sum_{n=0}^\infty \P[Z_n = 0] \right) \cdot \left( \sum_{i=x+1}^\infty \P[Y_1 \geq i]\right),\quad x > 0.\end{equation}
\item If $\rho = \E[Y_1] \neq 0$ and $\E[\exp(\kappa |Y_1|)] < \infty$ for some $\kappa > 0$, then for all $\varepsilon > 0$ there exists $c > 0$ such that
\begin{align}\label{eq:concentration}
&\P \left[ (\rho-\varepsilon) n \leq Z_n \leq (\rho + \varepsilon)n \right] > 1  - 2\exp(-cn),\quad n \in \mathbb{N} \text{ and }\\[.2cm]\label{eq:minimum_not_small}
&\P\left[  -\ell + (\rho-\varepsilon) n < Z_n < \ell + (\rho + \varepsilon) n \;\forall n\right] > 1- 2\exp(-c\ell),\quad \ell \geq 0.
\end{align}
\end{enumerate}
\end{lemma}
\begin{proof}
The second statement follows from standard large deviation estimates for random walks, so we will only prove the first one. We have:
\begin{align*}
\P\left[h^Z_{[\ell,\infty)} < \infty,\; Z_{h^Z_{[\ell,\infty)}} - \ell \geq x\right] & = \sum_{n=0}^\infty\; \sum_{y=-\infty}^{\ell -1} \P\left[h^Z_{[\ell,\infty)} = n +1,\; Z_n = y,\; Z_{n+1} - \ell \geq x \right]\\[.2cm]
&\leq \sum_{n=0}^\infty\;\sum_{y = -\infty}^{\ell - 1} \P[Z_n  = y]\cdot \P[Z_1 \geq x+\ell -y]\\
&\leq \sum_{i=x+1}^\infty \P[Z_1 \geq i] \cdot \sum_{n=0}^\infty \P[Z_n = 0]
\end{align*}
where the last inequality holds since, for all $y \in \Z$,
$$\E\left[ \#\{n: Z_n = y\}\right] \leq \E\left[ \#\{n: Z_n = 0\}\right]. $$
\end{proof}

Throughout this section, we will consider a sequence of independent and identically distributed random vectors
$$(X_1,\uptau_1),\;(X_2,\uptau_2),\ldots \in \Z \times (0,\infty)$$
satisfying
\begin{equation}\label{eq:assume_tail}
\E\left[\exp(\sigma \cdot |X_1|)\right] < \infty,\; \E\left[ \exp(\sigma\cdot \uptau_1)\right] < \infty \text{ for some } \sigma > 0
\end{equation}
and
\begin{equation}\label{eq:assume_drift}
\E[X_1] > 0.
\end{equation}
Taking in addition $(x_0,t_0) \in \Z \times \R$, we define a sequence $(S_n,T_n)_{n \geq 0}$ by letting $(S_0,T_0) = (x_0,t_0)$ and, for $n \geq 1$,
$$ T_n - T_{n-1} = \uptau_n,\qquad S_n - S_{n-1} = \begin{cases} X_n&\text{if } S_{n-1} \le 0;\\ -X_n&\text{if } S_{n-1} > 0,\end{cases}$$
so that $(T_n)$ is a renewal process and $(S_n)$ is a Markov chain on $\Z$ which on $\Z\backslash \{0\}$ has a drift in the direction of 0. Define $(S^+_n)$ by letting $S^+_0 = S_0 = x_0$ and 
\begin{equation}\label{eq:def_S+}
S^+_{n} = S^+_{n-1} + X_n,\quad n \geq 1. 
\end{equation}
For $A \subset \Z$, define the \textit{hitting times}
\begin{equation}\label{eq:def_hit}
h^S_A = \inf\{n \geq 0: S_n \in A\},\; h^{S^+}_A = \inf\{n \geq 0: S^+_n \in A\}, \; h^T_A = \inf\{n \geq 0: T_n \in A\}.
\end{equation}
Finally, let 
$$\mu = \E[X_1],\qquad \nu = \E[\uptau_1],\qquad \bar{\beta} = \mu/\nu.$$
Our goal is to prove:
\begin{proposition}\label{prop:cone}
For any $\beta < \bar{\beta}$ there exist $c > 0$ and $\ell_0 > 0$ such that the following holds. If $\ell \ge \ell_0$, $t \ge 0$, $|x_0| \le \beta t$, and $(S_0,T_0) = (x_0,0)$, then
$$\P\left[\left(S_{h^T_{[t,\infty)}}, T_{h^T_{[t,\infty)}}\right) \in [-\ell,\ell] \times [t, t+\ell]\;  \right] > 1-\exp(-c\ell). $$
\end{proposition}
In words: at the first time $n$ at which $T_n$ is above $t$, it is very likely that $(S_n, T_n)$ belongs to the box $[-\ell,\ell] \times [t,t+\ell]$. The proof of Proposition \ref{prop:cone} will depend on two preliminary results, Lemmas \ref{lem:right_line} and \ref{lem:straight}.

\begin{lemma}\label{lem:right_line}
For any $\beta < \bar{\beta}$, there exists $c > 0$ such that, if $S_0 = S^+_0 = T_0 = 0$, then
$$\P\left[S_n^+ \geq \beta T_n - \ell \; \forall n \right] > 1-2\exp(-c\ell),\qquad \ell > 0.$$
\end{lemma}
\begin{proof}
Given $\beta < \bar{\beta}$, choose $\mu' < \mu$ and $\nu' > \nu$ such that $\beta < \frac{\mu'}{\nu'} < \bar{\beta}$. By \eqref{eq:minimum_not_small}, there exists $c > 0$ such that for every $\ell$, with probability larger than $1-2\exp(-c\ell)$,
$$S_n^+ \geq \mu' n - \frac{\ell}{2}\quad \text{and}\quad T_n \leq \nu'n + \frac{ \nu'}{2 \mu'}\ell \quad \text{for all } n \geq 0. $$
If this occurs, then $$\beta T_n - \ell \leq \frac{\mu'}{\nu'}\left(\nu'n + \frac{\nu'}{2\mu'}\ell\right) - \ell \leq S_n^+$$ for every $n$.
\end{proof}

The following is a weaker version of Proposition \ref{prop:cone} which requires the initial position to be in the inner half of the interior of the spatial range of the target box.
\begin{lemma}\label{lem:straight}
There exists $ c > 0$ such that, for $\ell$ large enough and any $t \geq 0$,
\begin{equation}\begin{split}
&|S_0| \leq \frac{\ell}{2},\; T_0 = 0 \quad\Longrightarrow \quad\P\left[\left(S_{h^T_{[t,\infty)}}, T_{h^T_{[t,\infty)}}\right)  \in [-\ell,\ell]\times [t, t+\ell] \right] > 1-\exp(-c \ell).\end{split}
\end{equation}
\end{lemma}
Before proving this lemma, we will show how it can be combined with Lemma \ref{lem:right_line} (and the estimates of Lemma \ref{lem:large_deviations}) to prove Proposition \ref{prop:cone}.
\begin{proof}[Proof of Proposition \ref{prop:cone}]  Fix $\beta < \bar{\beta}$ and let $\ell$ be large enough as required in Lemma \ref{lem:straight}. Also let $t \geq 0$ and $x_0 \in \Z$ with $|x_0|\leq \beta\cdot t$. We will only treat the case where \begin{equation}x_0 \in [-\beta t,0];\label{eq:by_sym}\end{equation} the proof of the case $x_0 \in (0,\beta t]$ is entirely similar.  Throughout the proof, we will say an event occurs with high probability if its probability is larger than $1-\exp(-c\ell)$ for some $c > 0$ and $\ell$ large enough. 

Let
$$n^\star = \min\{h^T_{[t,\infty)},\;h^{S^+}_{[0,\infty)}\},$$
that is, $n^\star$ is the first time $n$ when we either have $T_n \geq t$ or $S_n^+ \geq 0$.
We will treat the two situations $n^\star = h^T_{[t,\infty)} < h^{S^+}_{[0,\infty)}$ and $n^\star = h^{S^+}_{[0,\infty)}\leq  h^T_{[t,\infty)}$ separately.

First assume that $n^\star = h^T_{[t,\infty)} < h^{S^+}_{[0,\infty)}$. Using \eqref{eq:by_sym}, we then have
$$T_{n^\star} \geq t,\qquad S_{n^\star} = S^+_{n^\star} \leq 0.$$
Moreover, using \eqref{eq:neat_form}, \eqref{eq:assume_tail}  and Chebyshev's inequality, with high probability we also have $T_{n^\star} \leq t+ \ell$. Additionally, by Lemma \ref{lem:right_line}, with high probability,
$$S_{n^\star} \geq x_0 - \ell + \beta \cdot T_{n^\star} \stackrel{\eqref{eq:by_sym}}{\geq} - \beta \cdot t - \ell +\beta \cdot T_{n^\star} \geq - \ell.$$
If all these conditions hold, then $(S_{n^\star},T_{n^\star}) \in [-\ell, \ell] \times [t, t+\ell]$.
This proves that
$$\P\left[h^T_{[t,\infty)} < h^{S^+}_{[0,\infty)},\; \left(S_{h^T_{[t,\infty)}}, T_{h^T_{[t,\infty)}}\right) \notin  [-\ell,\ell] \times [t,t+\ell]\right] < \exp(-c\ell).$$

We now turn to the case $n^\star = h^{S^+}_{[0,\infty)}\leq  h^T_{[t,\infty)}$. Again by \eqref{eq:by_sym}, we then have $S_{n^\star}\geq 0$, and by \eqref{eq:neat_form}, with high probability, we have $S_{n^\star} \leq \ell/2$. Then,  Lemma \ref{lem:straight} and the Markov property imply that, with high probability, $\left(S_{h^T_{[t,\infty)}}, T_{h^T_{[t,\infty)}}\right)  \in [-\ell,\ell] \times [t,\;t+\ell].$
Therefore,
$$\P\left[h^{S^+}_{[0,\infty)} \leq h^T_{[t,\infty)} ,\; \left(S_{h^T_{[t,\infty)}}, T_{h^T_{[t,\infty)}}\right) \notin [-\ell,\ell] \times [t,t+\ell]\right] < \exp(-c\ell),$$
completing the proof.
\end{proof}

We now turn to the proof of Lemma \ref{lem:straight}. We will need two more preliminary results, Lemmas \ref{lem:inside_interval} and \ref{lem:here_to_there}.

\begin{lemma} \label{lem:inside_interval} There exists $c > 0$ such that, if $m > 0$,
\begin{align}\begin{split}&|S_0|\leq m  \Longrightarrow   \P\left[|S_n| \leq 2m \; \forall n \leq m^6\right] > 1-\exp(-cm). \end{split}
\end{align}
\end{lemma}
\begin{proof} 
Assume first that $S_0 \in [-m, 0]$. Recall the definition of $(S^+_n)$ in \eqref{eq:def_S+} and note that $h^S_{(0,\infty)} = h^{S^+}_{(0,\infty)}$ and $S_n = S^+_n$ for $0 \leq n \leq h^S_{(0,\infty)}$. Hence,
\begin{align*}
\P\left[\min_{0 \leq n \leq h^S_{(0,\infty)}} S_n \geq -2m,\; S_{h^S_{(0,\infty)}} \leq m\right] &\geq 1-\P\left[ \min_{0 \leq n \leq h^{S^+}_{(0,\infty)}} S^+_n < -2m\right]-\P\left[S^+_{h^{S^+}_{(0,\infty)}} > m \right]\\
&\geq 1 - 2\exp(-cm)
\end{align*}
for some $c > 0$, by \eqref{eq:neat_form}, \eqref{eq:minimum_not_small} and \eqref{eq:assume_tail}. Similarly, if $S_0 \in (0,m]$ we have
$$\P\left[\max_{0 \leq n \leq h^S_{(-\infty,0]}} S_n \leq 2m,\; S_{h^S_{(-\infty,0]}} \geq -m \right] \geq 1-2\exp(-cm).$$
Applying these bounds together with the Markov property yields the desired result.
\end{proof}

\begin{lemma}\label{lem:here_to_there}
There exists $c > 0$ such that, if $m > 0$,
\begin{align}\label{eq:bar_to_large_box}
\begin{split}&S_0 \in [-m, m], \; T_0 = 0,\; 0 \leq t \leq m^5 \quad \\&\hspace{2cm}\Longrightarrow  \P\left[\left(S_{h^T_{[t,\infty)}}, T_{h^T_{[t,\infty)}}\right) \in [-2m,2m] \times [t, t+m]\right] > 1-2\exp(-cm);\end{split}\\[.2cm]\label{eq:bar_to_small_box}
\begin{split}&S_0 \in [-m, m], \; T_0 = 0,\; m^3 \leq t \leq m^5 \quad \\&\hspace{2cm}\Longrightarrow  \P\left[(S_{h^T_{[t,\infty)}}, T_{h^T_{[t,\infty)}}) \in [-m/2,m/2] \times [t, t+m]\right] > 1-2\exp(-cm).\end{split}
\end{align}
\end{lemma}
\begin{proof}
Assume that $(S_0, T_0) \in [-m,m]\times \{0\}$ and $0 \leq t \leq m^5$. Then, the event in the probability in \eqref{eq:bar_to_large_box} holds as soon as
$$h^T_{[t,\infty)} \leq m^6,\quad T_{h^T_{[t,\infty)} }-t \leq m,\quad \max_{0 \leq n \leq m^6} |S_n| \leq 2m. $$
By \eqref{eq:assume_tail} and Lemmas \ref{lem:large_deviations} and \ref{lem:inside_interval}, all these conditions hold with probability larger than $1-2\exp(-cm)$ for some $c > 0$, proving \eqref{eq:bar_to_large_box}.

We now turn to \eqref{eq:bar_to_small_box}. Assume first that $S_0 \in [-m,0]$. If we have
\begin{equation}
h^S_{(0,\infty)} \leq m^2,\quad h^T_{[m^3,\infty)} > m^2, \quad \text{and } S_{h^S_{(0,\infty)}} \leq \frac{m}{4},
\end{equation}
then there exists some $n^*$ such that $(S_{n^*}, T_{n^*}) \in [0,m/4]\times[0, m^3]$. Then, \eqref{eq:bar_to_small_box} follows from \eqref{eq:bar_to_large_box}. The case where $S_0 \in (0,m]$ is treated similarly.
\end{proof}

\begin{proof}[Proof of Lemma \ref{lem:straight}] Fix $\ell > 0$ and $t \geq 0$. We will present a construction consisting of disjoint space-time boxes labeled increasingly in time, so that with high probability the process $(S_n,T_n)_{n \geq 0}$ visits all of them and the last one is the `target' box in the statement of the lemma, $[-\ell,\ell] \times [t, t+\ell]$. A quick glimpse at Figure \ref{fig:boxes} will help the reader understand the construction.

We let $s_0 = t$ and
$$s_{i+1} = s_i - (\ell + i)^4,\qquad i = 0,1,\ldots.$$
Then define $k = \max\{i: s_i > 0\}$, and let $0=t_0 < t_1<\ldots < t_{k+1} = t$ be the values $s_0, \ldots, s_k, 0$ labeled in increasing order, that is, $t_0 = 0$ and $$t_i = s_{k+1-i},\qquad i = 1,\ldots, k+1.$$
Next, define the boxes
$$B_i = [-(\ell + k + 1 -i), \ell + k + 1-i] \times [t_i, \; t_i + \ell + k + 1 -i],\qquad i=1,\ldots, k+1.$$

Now, since $(S_0, T_0) \in [-\ell/2,\ell/2]\times \{0\}$, \eqref{eq:bar_to_large_box} implies that with probability larger than $1-\exp(-c(\ell + k))$, we have
$$\left(S_{h^T_{[t_1,\infty)}}, T_{h^T_{[t_1,\infty)}}\right) \in B_1.$$ Next, \eqref{eq:bar_to_small_box} implies that, for $n \in \N$ and $i \leq k$,
$$\P\left[ \left.\left(S_{h^T_{[t_{i+1},\infty)}},T_{h^T_{[t_{i+1},\infty)}}\right) \in B_{i+1} \right| (S_n, T_n) \in B_i\right] > 1-\exp(-c(\ell + k - i)). $$
Hence, with high probability, for every $i$, at the first $n$ for which we have $T_n > t_i$, we have $(S_n, T_n) \in B_i$. This completes the proof.

\begin{figure}[htb]
\begin{center}
\setlength\fboxsep{0pt}
\setlength\fboxrule{0.5pt}
\fbox{\includegraphics[width = 0.4\textwidth]{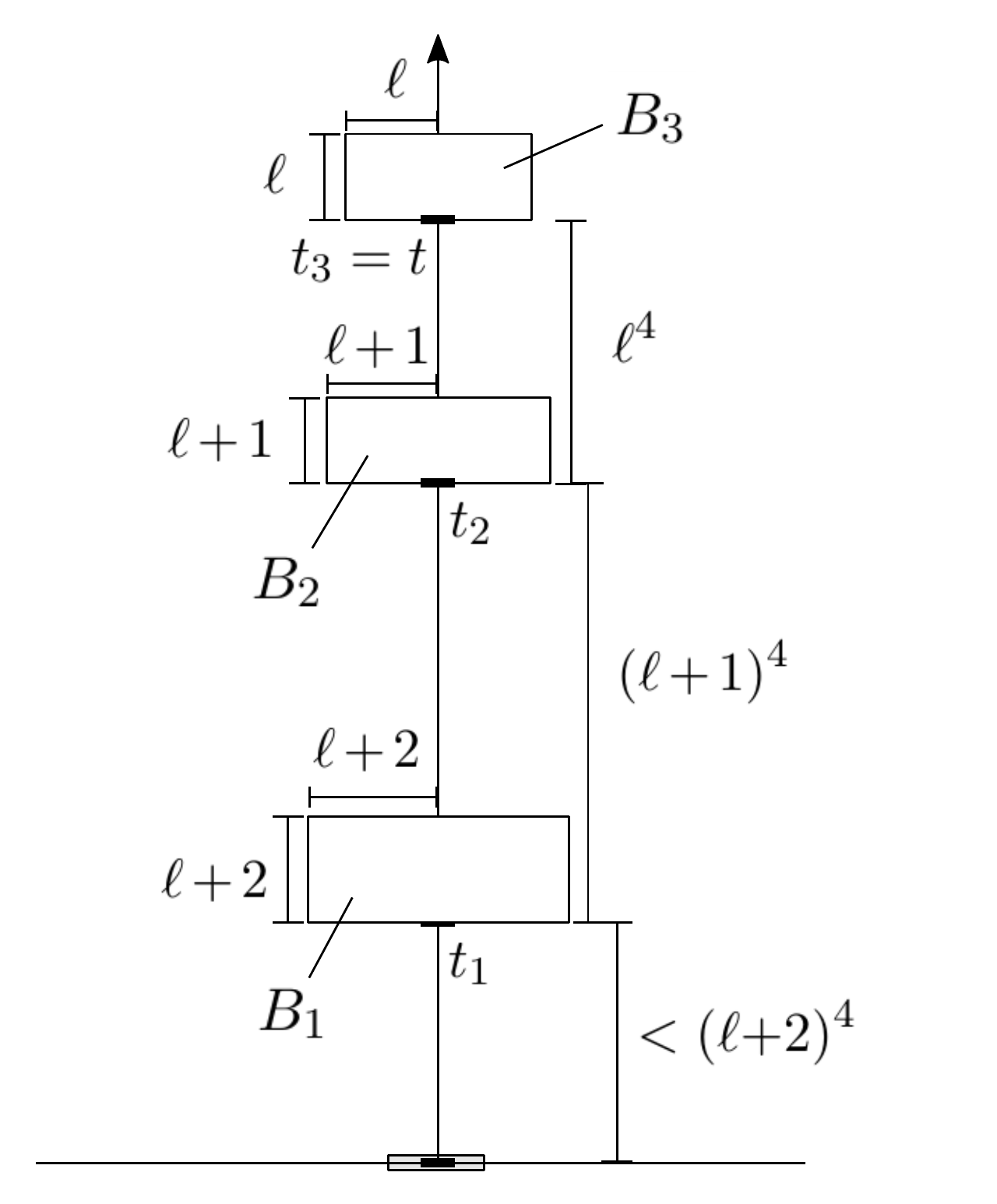}}
\end{center}
\caption{The initial position $(S_0,T_0)=(S_0,0)$ is located inside $[-\tfrac{\ell}{2},\tfrac{\ell}{2}] \times \{0\}$, the grey bar on the horizontal axis. With high probability, each of the boxes $B_i$ contains $(S_n,T_n)$ for some $n$. }
\label{fig:boxes}
\end{figure}

\end{proof}

%%%%%%%%%%%% References %%%%%%%%%%%%%%%%%%%%%%%%%%%%%%%%%%%%%%%%%%%%%%%%%%%%%%%%%%%%%%%%%%%%%%%%%%%%%%%%%%%%%%%%%%%%%%%


\newpage
\begin{thebibliography}{99}
\bibitem{amp} E. Andjel, J. Miller, E. Pardoux, \textit{Survival of a single mutant in one dimension}, Electronic Journal of Probability 15, 386-408 (2010).

\bibitem{ampv} E. Andjel, T. Mountford, L. P. R. Pimentel, D. Valesin, \textit{Tightness for the Interface of the One-Dimensional Contact Process}, Bernoulli 16, Number 4 (2010).

\bibitem{bezui} C. Bezuidenhout, G. Grimmett,  \textit{The critical contact process dies out},. Ann. Probability 4 (1990).

\bibitem{bill} P. Billingsley, \textit{Convergence of probability measures}, Second edition. Wiley Series in Probability and Statistics: Probability and Statistics (1999).

\bibitem{durschon} R. Durrett, R. H. Schonmann, \textit{Large deviations for the contact process and two dimensional percolation}. Probability theory and related fields 77(4), pp.583-603 (1988).

\bibitem{harris} T. E. Harris, \textit{Contact interactions on a lattice}, Ann. Probability 2 (1974).

\bibitem{lawler} G. Lawler, V. Limic, \emph{Random Walk: A Modern Introduction}, Cambridge University Press  (2010).


\bibitem{lig85}
T. Liggett, \emph{Interacting Particle Systems}, Grundlehren der Mathematischen
Wissenschaften 276, Springer, New York (1985).

\bibitem{lig99}
T. Liggett, \emph{Stochastic Interacting Systems: Contact, Voter and Exclusion Processes}, Grundlehren der Mathematischen Wissenschaften 324, Springer, Berlin (1999).

\bibitem{interface} T. Mountford, D. Valesin, \emph{Functional Central Limit Theorem for the Interface of the Symmetric Multitype Contact Process,} ALEA 13, no. 1: 481-519 (2016).

\bibitem{neuhauser}
C. Neuhauser, \textit{Ergodic Theorems for the Multitype Contact Process,} Probability Theory and Related Fields 91, 467-506 (1992).

\bibitem{spitzer} F. Spitzer, \textit{Principles of Random Walk}, 2nd edition, New York, NY: Springer-Verlag, (2001).

\bibitem{valesin} D. Valesin, \textit{Multitype Contact Process on $\Z$: Extinction and Interface}, Electronic Journal of Probability 15, 2220-2260 (2010).
\end{thebibliography}
\end{document}